\documentclass[11pt]{article}
\usepackage{enumitem}

\usepackage{orcidlink}

\usepackage{amsmath}
\usepackage{amsthm}
\usepackage{amssymb}
\usepackage{bussproofs}
\usepackage{stmaryrd}
\usepackage{geometry}
\usepackage{stix}
\DeclareSymbolFont{LMletters}{OML}{lmm}{m}{it}
\SetSymbolFont{LMletters}{bold}{OML}{lmm}{b}{it}

\DeclareMathSymbol{\mi}{\mathalpha}{LMletters}{`i}
\DeclareMathSymbol{\mj}{\mathalpha}{LMletters}{`j}
\DeclareMathSymbol{\mk}{\mathalpha}{LMletters}{`k}
\DeclareMathSymbol{\ml}{\mathalpha}{LMletters}{`l}
\DeclareMathSymbol{\mr}{\mathalpha}{LMletters}{`r}

\usepackage{cleveref}
\usepackage{mathtools}
\usepackage{url}
\usepackage{float}

\usepackage{thmtools, thm-restate}

\usepackage{enumitem}

\makeatletter
\DeclareRobustCommand{\rvdots}{%
  \vbox{
    \baselineskip4\p@\lineskiplimit\z@
    \kern-\p@
    \hbox{.}\hbox{.}\hbox{.}
  }}
\makeatother

\theoremstyle{plain}
\newtheorem{theorem}{Theorem}
\newtheorem{lemma}[theorem]{Lemma}
\newtheorem{proposition}[theorem]{Proposition}
\newtheorem{corollary}[theorem]{Corollary}

\theoremstyle{definition}
\newtheorem{definition}[theorem]{Definition}

\newtheorem*{remark}{Remark}

\newtheorem{claim}{Claim}[theorem]

\newenvironment{claimproof}[1][\proofname]{%
  \begin{proof}[#1]%
}{%
  \end{proof}%
}

\providecommand{\keywords}[1]
{
  \small	
  \textbf{\textit{Keywords:}} #1
}

\newcommand{\reptitle}{}
\theoremstyle{plain}
\newtheorem*{reptheorem}{\reptitle}
\theoremstyle{definition}

\title{Constructive proofs for the standard translation of many-sorted to unsorted predicate logic}
\author{Hrafn Valtýr Oddsson\thanks{Institut für Philosophie I, Ruhr-Universit\"at Bochum. Email: hrafn.oddsson@rub.de} \ \orcidlink{0000-0003-1594-340X}}
\date{}

\begin{document}

\maketitle

\begin{abstract}
It is well known that many-sorted logic can be reduced to unsorted first-order logic by adding predicates for each sort, relativizing quantifiers to these predicates, and adding appropriate axioms governing their behavior. Existing constructive proofs for the correctness of this translation break down when the many-sorted language includes equality and the unsorted target calculus includes the usual rules/axioms for equality. We give an elementary proof in the form of an effective procedure that closes this gap. As an application, we give a fully syntactic justification of van Dalen's translation of second-order logic into unsorted first-order logic.

We also give a new proof for a claim made by Herbrand in his 1930 dissertation that, in the equality-free case, a sentence is derivable in many-sorted logic iff it is derivable in unsorted logic. Our proof avoids the heavy machinery of later proofs by Schmidt and Wang.

\end{abstract}

\keywords{Intuitionistic logic; many-sorted logic; multi-sorted logic; second-order logic}

\section{Introduction} 
Loosely speaking, many-sorted logic ($\mathrm{MSL}$) extends first-order ($\mathrm{FOL}$) logic by allowing quantification over different sorts of objects. For example, one can formalize geometry by having variables of one sort range over points, another over lines, yet another over circles, and so on.

\smallskip

The standard way of reducing many-sorted logic to unsorted is to introduce a predicate symbol $Q_\mi$ for each sort $\mi$. Then $Q_\mi t$ expresses that $t$ is of sort $\mi$. Instead of $\exists x^\mi(\ldots)$ and $\forall x^\mi(\ldots)$, one writes $\exists x(Q_\mi x\land\ldots)$ and $\forall x(Q_\mi x\to\ldots)$. This way, we obtain the translation $\varphi^*$ of a sentence $\varphi$. To go with this, we also need to add a set $J$ of sort axioms ensuring, at minimum, that each $Q_\mi$ is nonempty and that the constant and function symbols behave correctly with respect to the intended sorts. The correctness of this translation is then usually stated as follows:
$$\Gamma\vdash_\mathrm{MSL}\varphi \quad\text{iff} \quad \Gamma^*, J\vdash_\mathrm{FOL}\varphi^* .$$

Giving a semantic proof of this equivalence is relatively straightforward; see, e.g., \cite{ManzanoAranda2022}. Moreover, it is well understood how to give an effective procedure that takes a derivation of $\Gamma\vdash_\mathrm{MSL}\varphi$ and produces a derivation of $\Gamma^*, J\vdash_\mathrm{FOL}\varphi^*$. However, providing an effective procedure that takes a derivation in the one-sorted logic and produces a proof in the many-sorted logic has proved more intricate. In his PhD thesis \cite{THESE_1930__110__1_0}\footnote{See \cite{Herbrand1971} for an English translation.}, Herbrand sketches an argument along the following lines: Take a derivation of $\varphi^*$ from $\Gamma^*\cup J$ and replace each instance of $Q_\mi t$ with a tautology. This procedure transforms $\varphi^*$ into a sentence equivalent to $\varphi$, $\Gamma^*$ into a set of sentences equivalent to $\Gamma$, and $J$ into a set of tautologies. The result is a derivation of $\varphi$ from $\Gamma$ in the unsorted logic. He then concludes $\Gamma\vdash_\mathrm{MSL}\varphi$. In effect, he is claiming that
$$\Gamma\vdash_\mathrm{MSL}\varphi\quad\text{iff}\quad \Gamma\vdash_\mathrm{FOL}\varphi.$$

In \cite{Schmidt1938}, Schmidt points out that this step in Herbrand's argument is not justified. The reason being that a derivation on the $\mathrm{FOL}$ side may go through formulas that are not themselves well-formed many-sorted formulas, so the resulting derivation is not a derivation in the many-sorted logic. He therefore gives a rather involved procedure intended to transform an unsorted derivation into one containing only well-formed many-sorted formulas. However, this procedure falls short, and he later gives a corrected version \cite{Schmidt1951}. In the end, he gets a derivation of $\varphi$ from $\Gamma$ that can be read as a derivation in the many-sorted logic, thus establishing Herbrand's claim. 

In \cite{Wang1952}, Wang takes a similar approach to Schmidt. The main difference is that Wang uses Herbrand's theorem\footnote{"Herbrand's theorem" here refers to the fact that any derivable sentence in prenex normal form has a derivation of a certain shape.} to first obtain a one-sorted derivation, which is easier to transform into one that only contains well-formed many-sorted formulas. 

\smallskip

It is important to note at this point that this approach of replacing all instances of sort predicates with a tautology is incompatible with equality with its usual rules/axioms in the one-sorted logic. To see why, consider the formula $\forall x^\mi\forall y^\mi(x^\mi=y^\mi)$. It expresses that there is only one object of sort $\mi$. It does not follow that there is only one object of some other sort $\mk$. However, $$\big( \forall x^\mi\forall y^\mi(x^\mi=y^\mi)\big)^*=\forall x(Q_\mi x\to \forall y(Q_\mi y\to x=y)).$$ If we replace $Q_\mi x$ and $Q_\mi y$ with $\top$, we get a formula that is equivalent to $\forall x\forall y(x=y)$. Similarly, if we add additional axioms to $J$ expressing that the sorts are mutually disjoint, such as $ \forall x\neg(Q_\mi x\land Q_\mj x)$, then replacing the sort predicates with $\top$ would result in a contradiction. In either case, the proofs by Schmidt and Wang break down.

\smallskip

It is also worth mentioning that Gilmore \cite{Gilmore1958} points out that Wang's proof requires all relation symbols to be \emph{strict}, in the sense that every argument place is assigned a fixed sort. So, the proof does not cover cases where relation symbols can accept terms of different sorts in the same argument place. For example, Wang treats \emph{the simple theory of types} where formulas of the form $x^n\epsilon y^{n+1}$ are considered well-formed for any $n$. So, under Wang's restriction, one would need infinitely many membership predicates $\epsilon_n$, one for each $n\in \mathbb{N}$. Gilmore provided a proof that covers these cases. Still, his proof does not handle equality. 

\smallskip

In \cite{vanDalen1980}, van Dalen provides a translation of second-order logic into unsorted first-order logic. His translation can be decomposed into two steps. First, translate second-order logic into many-sorted first-order logic with one sort for individuals and one sort for $n$-ary predicates for each $n$, together with relation symbols $Ap_n$ where $Ap_n(X^n, x_1, \dots, x_n)$ expresses that the predicate $X^n$ holds of the individuals $x_1, \dots, x_n$. Next, apply the standard reduction from many-sorted to unsorted first-order logic with equality, together with axioms stating that the sort predicates are pairwise disjoint. He uses this translation to establish that his proof system is complete with respect to the given semantics. He does not provide a proof of correctness for the translation, stating that it is a ``tedious but routine job.''

Nour and Raffalli point out in \cite{NourRaffalli2003} that this task is not as routine as van Dalen makes it out to be, writing: ``It is not even clear that the proof in [12] which is only sketched can be completed into a correct proof (at least the authors do not know how to end his proof).''\footnote{Here, ``[12]'' is \cite{vanDalen1994}.} They note that formulas appearing in a derivation of $\varphi^*$ are not necessarily of the form $\psi^*$. They therefore opt for an alternative translation: rather than adding sort predicates, they let the placement of a variable in atomic formulas determine its role. In this way, one variable can play different roles in the same formula. This allows them to give a reverse coding $A^\diamond$ for any first-order $A$ in such a way that $\varphi$ and $\varphi^{*\diamond}$ are equivalent. 

\smallskip

In this paper, we give two constructive proofs, both of which work for classical and intuitionistic logic. The first proof establishes the correctness of the standard translation in the presence of equality, as well as relation and function symbols that allow more than one combination of sorts in their argument places. We then extend this result to accommodate additional sort axioms, including those expressing that the sorts are mutually disjoint. As an application, we use it to give a fully syntactic justification of van Dalen's translation of second-order logic into unsorted first-order logic. Our second proof adapts the argument of Nour and Raffalli to $\mathrm{MSL}$ and gives a new proof of Herbrand's claim that
$$\Gamma\vdash_\mathrm{MSL}\varphi \quad\text{iff} \quad \Gamma\vdash_\mathrm{FOL}\varphi,$$
whenever $\Gamma\cup\{\varphi\}$ is a set of equality-free $\mathrm{MSL}$-sentences with only strict relation and function symbols. This proof gives a simpler alternative to the ones by Schmidt and Wang.

\section{Preliminaries}
We introduce both the single- and many-sorted logics and the standard translation between them. Our notion of many-sorted logic is slightly more general than usual. In particular, we allow both relation and function symbols that admit different combinations of sorts in their argument places. We also consider two notions of equality: a strict one, where only terms of the same sort can be compared, and a liberal one, where terms of different sorts may be compared. Each of these features appears in the literature (see \cite{Gilmore1958}, \cite{GOGUEN1992217}, and \cite{ManzanoAranda2022}, respectively).

\subsection{$\mathrm{FOL}$}

\begin{definition}
The logical signs of $\mathrm{FOL}$ are:
\begin{itemize}
    \item a countably infinite set of variables $\mathrm{Var}:=\{x_0, x_1,\ldots\}$;
    \item the connectives $\land,\,\lor,\,\to,\,$ and $\bot$; 
    \item the quantifiers $\exists$ and $\forall$;
    \item the equality sign $=$.
\end{itemize}

\noindent An \emph{unsorted signature} $L$ consists of a countable set $\mathrm{Rel}$ of relation symbols and a countable set $\mathrm{Fun}$ of function symbols, each equipped with an arity $n\geq 0$. Relation symbols of arity $0$ are called \emph{propositional variables}, and function symbols of arity $0$ are called \emph{constants}. 

The notions of \emph{$L$-terms}, \emph{$L$-formulas}, and \emph{$L$-sentences} are defined as usual. We use upper-case letters from the beginning of the alphabet, such as $A$, $B$, and $C$, to refer to formulas and sentences. We write $\neg A$ for $A\to\bot$, $\top$ for $\bot\to\bot$, and $A\leftrightarrow B$ for $(A\to B)\land(B\to A)$.

If $t_1,\dots,t_n$ are terms and $x_1,\dots, x_n$ are variables, then we denote the result of simultaneously replacing each free occurrence of $x_k$ with $t_k$ in a term $t$ by $t[t_1/x_1,\dots,t_n/x_n]$ and in a formula $A$ by $A[t_1/x_1,\dots,t_n/x_n]$. Whenever we write $A[t_1/x_1,\dots,t_n/x_n]$, we assume that each $t_k$ is free for $x_k$ in $A$.
\end{definition}

We write $\Pi\vdash_{\mathrm{FOL}} A$ to indicate that $A$ is derivable from $\Pi$ using the rules from  Table~\ref{FOLrules} together with (the universal closure of) the following axioms for equality:

\begin{itemize}
    \item $x= x$
    \item $x=y\to y=x$
    \item $x=y\land y=z\to x=z$
    \item $x_1= y_1\land \ldots \land x_n= y_n\rightarrow \left(R(x_1,\ldots,x_n)\to R(y_1,\ldots,y_n)\right)$
    \item $x_1= y_1\land \ldots \land x_n= y_n\rightarrow f(x_1,\ldots,x_n)= f(y_1,\ldots,y_n)$
\end{itemize}

\begin{remark}
We work with classical logic throughout this paper. For the intuitionistic case, simply remove $RAA$ from Table~\ref{FOLrules} and all the results remain valid.
\end{remark}

\begin{table}
\centering
\caption{Natural deduction rules for $\mathrm{FOL}$}
\label{FOLrules}
\vspace{1em}
\footnotesize\begin{tabular}{ccc}

\begin{minipage}{0.26\textwidth}
\begin{prooftree}
\AxiomC{$A$}
\AxiomC{$B$}
\RightLabel{${\wedge} I$}
\BinaryInfC{$A \wedge B$}
\end{prooftree}
\end{minipage} &
\begin{minipage}{0.26\textwidth}
\begin{prooftree}
\AxiomC{$A \wedge B$}
\RightLabel{${\wedge} E_1$}
\UnaryInfC{$A$}
\end{prooftree}
\end{minipage} &
\begin{minipage}{0.26\textwidth}
\begin{prooftree}
\AxiomC{$A \wedge B$}
\RightLabel{${\wedge} E_2$}
\UnaryInfC{$B$}
\end{prooftree}
\end{minipage} \\[3em]

\begin{minipage}{0.26\textwidth}
\begin{prooftree}
\AxiomC{$A$}
\RightLabel{${\vee} I_1$}
\UnaryInfC{$A \vee B$}
\end{prooftree}
\end{minipage} &
\begin{minipage}{0.26\textwidth}
\begin{prooftree}
\AxiomC{$B$}
\RightLabel{${\vee} I_2$}
\UnaryInfC{$A \vee B$}
\end{prooftree}
\end{minipage} &
\begin{minipage}{0.26\textwidth}
\begin{prooftree}
\AxiomC{$A \vee B$}
\AxiomC{$[A]$}
\noLine
\UnaryInfC{$\vdots$}
\noLine
\UnaryInfC{$C$}
\AxiomC{$[B]$}
\noLine
\UnaryInfC{$\vdots$}
\noLine
\UnaryInfC{$C$}
\RightLabel{${\vee} E$}
\TrinaryInfC{$C$}
\end{prooftree}
\end{minipage} \\[3em]

\begin{minipage}{0.26\textwidth}
\begin{prooftree}
\AxiomC{$[A]$}
\noLine
\UnaryInfC{$\vdots$}
\noLine
\UnaryInfC{$B$}
\RightLabel{${\to} I$}
\UnaryInfC{$A \to B$}
\end{prooftree}
\end{minipage} &
\begin{minipage}{0.26\textwidth}
\begin{prooftree}
\AxiomC{$A \to B$}
\AxiomC{$A$}
\RightLabel{${\to} E$}
\BinaryInfC{$B$}
\end{prooftree}
\end{minipage} &
\begin{minipage}{0.26\textwidth}
\begin{prooftree}
\AxiomC{$\bot$}
\RightLabel{${\bot} E$}
\UnaryInfC{$A$}
\end{prooftree}
\end{minipage} \\[3em]

\begin{minipage}{0.26\textwidth}
\begin{prooftree}
\AxiomC{$A$}
\RightLabel{${\forall} I$}
\UnaryInfC{$\forall x\, A$}
\end{prooftree}
\end{minipage} &
\begin{minipage}{0.26\textwidth}
\begin{prooftree}
\AxiomC{$\forall x\, A$}
\RightLabel{${\forall} E$}
\UnaryInfC{$A[t/x]$}
\end{prooftree}
\end{minipage} &
\begin{minipage}{0.26\textwidth}
\begin{prooftree}
\AxiomC{$[\neg A]$}
\noLine
\UnaryInfC{$\vdots$}
\noLine
\UnaryInfC{$\bot$}
\RightLabel{$RAA$}
\UnaryInfC{$A$}
\end{prooftree}
\end{minipage} \\[3em]

\begin{minipage}{0.26\textwidth}
\begin{prooftree}
\AxiomC{$A[t/x]$}
\RightLabel{${\exists} I$}
\UnaryInfC{$\exists x\, A$}
\end{prooftree}
\end{minipage} &
\begin{minipage}{0.26\textwidth}
\begin{prooftree}
\AxiomC{$\exists x\, A$}
\AxiomC{$[A]$}
\noLine
\UnaryInfC{$\vdots$}
\noLine
\UnaryInfC{$C$}
\RightLabel{${\exists} E$}
\BinaryInfC{$C$}
\end{prooftree}
\end{minipage} &
\begin{minipage}{0.26\textwidth}
\end{minipage} \\[2em]

\end{tabular}

\medskip

\begin{minipage}{0.88\textwidth}
\footnotesize
In $\forall I$ and $\exists E$, the variable $x$ may not occur free in any
undischarged assumption. In $\exists E$, $x$ may also not occur free in $C$.
\end{minipage}
 
\end{table}

\subsection{$\mathrm{MSL}$}

The logical signs of $\mathrm{MSL}$ are the same as those of $\mathrm{FOL}$: the set of variables $\mathrm{Var}$, the connectives $\land$, $\lor$, $\to$, $\bot$, the quantifiers $\exists$ and $\forall$, and the equality sign $=$.\footnote{We treat liberal equality in a later section.}

\begin{definition}\label{def:signature}
A \emph{many-sorted signature} $\mathcal{L}$ consists of an unsorted signature $L$, together with:
\begin{itemize}
    \item a nonempty countable set $\mathrm{Sort}$ of sorts;
    \item pairwise disjoint countably infinite sets $\mathrm{Var}_\mi\subseteq \mathrm{Var}$ of \emph{variables of sort $\mi$}, one for each $\mi\in \mathrm{Sort}$; and
    \item a function $\mathrm{rank}$ that assigns:
    \begin{itemize}
        \item to each $n$-ary relation symbol $R\in L$, a nonempty set $\mathrm{rank}(R)\subseteq \mathrm{Sort}^n$,
        \item to each $n$-ary function symbol $f\in L$, a nonempty partial function $\mathrm{rank}(f)\subseteq \mathrm{Sort}^n\times \mathrm{Sort}$. That is, for each $(\mi_1,\ldots ,\mi_n)\in \mathrm{Sort}^n$, there is at most one $\mi\in \mathrm{Sort}$ such that $(\mi_1,\ldots \mi_n,\mi)\in\mathrm{rank}(f)$. We denote this sort by $\mathrm{rank}(f)(\mi_1,\dots,\mi_n)$.
    \end{itemize}
\end{itemize}
We write $x^\mi$ to indicate that $x\in \mathrm{Var}_\mi$, and we set $\mathrm{Var}_{\mathcal{L}}:=\bigcup_{\mi\in\mathrm{Sort}}\mathrm{Var}_\mi$.
\end{definition}

\begin{definition}\label{def:terms}
The \emph{well-sorted terms} of $\mathcal{L}$ and their sorts are defined as follows:
\begin{enumerate}
    \item Each variable $x\in \mathrm{Var}_\mi$ is a well-sorted term of sort $\mi$.
    \item If $f$ is an $n$-ary function symbol, $t_1,\ldots, t_n$ are well-sorted terms of sorts $\mi_1,\ldots,\mi_n$ respectively, and $(\mi_1,\ldots,\mi_n,\mi)\in \mathrm{rank}(f)$, then $f(t_1,\ldots,t_n)$ is a well-sorted term of sort $\mi$.
\end{enumerate}
We write $\mathrm{sort}(t)$ for the sort of a well-sorted term $t$, and $t^\mi$ to indicate that $\mathrm{sort}(t)=\mi$.
\end{definition}

\begin{definition}\label{def:formulas}
The \emph{well-sorted formulas} of $\mathcal{L}$ are defined as follows:
\begin{enumerate}
    \item $\bot$ is well sorted.
    \item If $t_1$ and $t_2$ are terms with the same sort, then $t_1=t_2$ is well sorted.
    \item If $R$ is an $n$-ary relation symbol, $t_1,\ldots, t_n$ are well-sorted terms of sorts $\mi_1,\ldots,\mi_n$ respectively, and $(\mi_1,\ldots,\mi_n)\in \mathrm{rank}(R)$, then $R(t_1,\ldots,t_n)$ is well sorted.
    \item If $\varphi$ and $\psi$ are well sorted, then so is $\varphi\star\psi$ for $\star\in\{\land,\lor,\to\}$.
    \item If $\varphi$ is well sorted, then so are $\exists x^\mi\varphi$ and $\forall x^\mi\varphi$.
\end{enumerate}
A \emph{well-sorted sentence} is a well-sorted formula with no free variables. By \emph{$\mathcal{L}$-formulas},  \emph{$\mathcal{L}$-sentences}, and \emph{$\mathcal{L}$-terms}, we mean well-sorted ones.
\end{definition}

We write $\Gamma\vdash_{\mathrm{MSL}}^{\mathcal{L}}\varphi$ to indicate that $\varphi$ is derivable from $\Gamma$ using the propositional rules from Table~\ref{FOLrules} and the quantifier rules from Table~\ref{MSLrules}, restricted to well-sorted $\mathcal{L}$-formulas, together with the (universal closure of the) following axioms:
\begin{itemize}
    \item $x^\mi= x^\mi$
    \item $x^\mi= y^\mi\to y^\mi=x^\mi$
    \item $x^\mi= y^\mi\land y^\mi= z^\mi\to x^\mi= z^\mi$
    \item $x_1^{\mi_1}= y_1^{\mi_1}\land \ldots \land x_n^{\mi_n} = y_n^{\mi_n}\rightarrow \left(R(x_1^{\mi_1},\ldots,x_n^{\mi_n})\to R(y_1^{\mi_1},\ldots,y_n^{\mi_n})\right)$, whenever $(\mi_1,\ldots ,\mi_n)\in \mathrm{rank}(R)$
    \item $x_1^{\mi_1}= y_1^{\mi_1}\land \ldots \land x_n^{\mi_n} = y_n^{\mi_n}\rightarrow f(x_1^{\mi_1},\ldots,x_n^{\mi_n})= f(y_1^{\mi_1},\ldots, y_n^{\mi_n})$, whenever $(\mi_1,\ldots ,\mi_n)\in \mathrm{dom}\left(\mathrm{rank}(f)\right)$
\end{itemize}

\begin{table}
\centering
\caption{Quantifier rules for $\mathrm{MSL}$}
\label{MSLrules}
\vspace{1em}
\footnotesize\begin{tabular}{ccc}

\begin{minipage}{0.26\textwidth}
\begin{prooftree}
\AxiomC{$\varphi$}
\RightLabel{${\forall} I$}
\UnaryInfC{$\forall x^\mi\, \varphi$}
\end{prooftree}
\end{minipage} &
\begin{minipage}{0.26\textwidth}
\begin{prooftree}
\AxiomC{$\forall x^\mi\, \varphi$}
\RightLabel{${\forall} E$}
\UnaryInfC{$\varphi[t^\mi/x^\mi]$}
\end{prooftree}
\end{minipage} &
\begin{minipage}{0.26\textwidth}
\end{minipage} \\[3em]

\begin{minipage}{0.26\textwidth}
\begin{prooftree}
\AxiomC{$\varphi[t^\mi/x^\mi]$}
\RightLabel{${\exists} I$}
\UnaryInfC{$\exists x^\mi\, \varphi$}
\end{prooftree}
\end{minipage} &
\begin{minipage}{0.26\textwidth}
\begin{prooftree}
\AxiomC{$\exists x^\mi\, \varphi$}
\AxiomC{$[\varphi]$}
\noLine
\UnaryInfC{$\vdots$}
\noLine
\UnaryInfC{$\psi$}
\RightLabel{${\exists} E$}
\BinaryInfC{$\psi$}
\end{prooftree}
\end{minipage} &
\begin{minipage}{0.26\textwidth}
\end{minipage} \\[2em]

\end{tabular}

\medskip

\begin{minipage}{0.8\textwidth}
\footnotesize
In $\forall I$ and $\exists E$, the variable $x^\mi$ may not occur free in any
undischarged assumption. In $\exists E$, $x^\mi$ may also not occur free in $\psi$.
\end{minipage}
\end{table}

\bigskip

\begin{proposition}\label{SimInst}\mbox{}
    \begin{enumerate}
        \item[(i)] $\varphi_1\leftrightarrow\varphi_2,\,\chi[\varphi_1/p]\vdash_\mathrm{MSL}   \chi[\varphi_2/p]$, where $p$ is a propositional variable and $\chi[\varphi/p]$ is the result of replacing each instance of $p$ with $\varphi$;
        \item[(ii)] $t_1^\mi=t_2^\mi,\,\varphi[t_1^\mi/x^\mi]\vdash_\mathrm{MSL}\varphi[t_2^\mi/x^\mi]$.
    \end{enumerate}
\end{proposition}
\begin{proof}
    The proofs are entirely standard and therefore omitted.
    
\end{proof}

\begin{definition}\label{def:strict-liberal}
An $n$-ary relation symbol $R$ is called \emph{strict} if $\mathrm{rank}(R)$ is a singleton, and \emph{liberal} if $\mathrm{rank}(R)=\mathrm{Sort}^n$. Similarly, an $n$-ary function symbol $f$ is called \emph{strict} if $\mathrm{rank}(f)$ is a singleton, and \emph{liberal} if $\mathrm{dom}(\mathrm{rank}(f))=\mathrm{Sort}^n$. For strict relation and function symbols, we will abuse notation and write $\mathrm{rank}(R)=(\mi_1,\dots \mi_n)$ and $\mathrm{rank}(f)=(\mi_1,\dots \mi_n,\mi)$, respectively.
\end{definition}

\begin{proposition}\label{StrictIsSufficient}
    Fix a signature $\mathcal{L}$, and let $\mathcal{L}^s$ be the signature obtained by replacing each $R\in\mathrm{Rel}$ and $f\in\mathrm{Fun}$ with strict symbols $R_{(\mi_1,\dots,\mi_n)}$ and $f_{(\mi_1,\dots,\mi_n)}$, one for each sort in $\mathrm{rank}(R)$ and $\mathrm{rank}(f)$, respectively. Let $\varphi^s$ be obtained from $\varphi$ by replacing each occurrence of a relation or function symbol with the associated strict symbol from $\mathcal{L}^s$. Then
    $$\Gamma\vdash_\mathrm{MSL}^{\mathcal{L}}\varphi \quad\text{iff}\quad \Gamma^s\vdash_\mathrm{MSL}^{\mathcal{L}^s}\varphi^s.$$
\end{proposition}
\begin{proof}
    Both directions are routine inductions on the length of the derivation.
    
\end{proof}

\subsection{Signature extensions}\label{SigExt}

\begin{definition}\label{def:extension}
    Let $\mathcal{L}$ and $\mathcal{L}'$ be many-sorted signatures. We say that $\mathcal{L}'$ \emph{extends} $\mathcal{L}$, and write $\mathcal{L}\leq \mathcal{L}'$, if:
    \begin{enumerate}
        \item $\mathrm{Sort}\subseteq \mathrm{Sort}'$;
        \item $\mathrm{Var}_\mi=\mathrm{Var}'_\mi$ for each $\mi\in \mathrm{Sort}$;
        \item $\mathrm{Rel}\subseteq \mathrm{Rel}'$ and $\mathrm{Fun}\subseteq \mathrm{Fun}'$;
        \item $\mathrm{rank}(R)\subseteq \mathrm{rank}'(R)$ for all $R\in \mathrm{Rel}$;
        \item $\mathrm{rank}(f)\subseteq \mathrm{rank}'(f)$ for all $f\in \mathrm{Fun}$.
    \end{enumerate}
 In particular, every well-sorted $\mathcal{L}$-term is a well-sorted $\mathcal{L}'$-term with the same sort, and every well-sorted $\mathcal{L}$-formula is a well-sorted $\mathcal{L}'$-formula.
\end{definition}

\begin{proposition}\label{ConLan}
If $\mathcal{L}\leq \mathcal{L}'$ and $\Gamma\cup\{\varphi\}$ is a set of $\mathcal{L}$-formulas, then
        $$\Gamma\vdash_{\mathrm{MSL}}^{\mathcal{L}'}\varphi \quad\text{implies}\quad \Gamma\vdash_{\mathrm{MSL}}^{\mathcal{L}}\varphi.$$
\end{proposition}

\begin{proof} Note that $\mathcal{L}\leq \mathcal{L}'$ iff $\mathcal{L}^s\leq (\mathcal{L}')^s$. So, without loss of generality, we can assume that $\mathcal{L}$ and $\mathcal{L}'$ only have strict relation and function symbols. Fix an $\mathcal{L}'$-derivation $\mathcal{D}$ of $\varphi$ from $\Gamma$, and for each sort $\mi\in\mathrm{Sort}'$, fix a fresh variable $z_\mi$. Given an $\mathcal{L}'$-term $t$, we define $t^\circ$ by letting $x^\circ:=x$ and
$$
\bigl(f(t_1,\dots,t_n)\bigr)^\circ \;:=\;
  \begin{cases}
    f(t_1^\circ,\dots,t_n^\circ) & \text{if $f\in \mathrm{Fun}$,} \\
    z_{\mathrm{sort}(f(t_1,\dots,t_n))} & \text{else.}
  \end{cases}
$$
If $t$ is an $\mathcal{L}$-term, then $t^\circ=t$. A simple induction also gives $\mathrm{sort}(t^\circ)=\mathrm{sort}(t)$ for every $\mathcal{L}'$-term $t$. Moreover, since each function symbol is strict, we get that $t^\circ$ is an $\mathcal{L}$-term when $\mathrm{sort}(t)\in\mathrm{Sort}$.

\smallskip

We fix a fresh variable $w^{\mi_0}$ with $\mi_0\in\mathrm{Sort}$ and further define $(\cdot)^\circ$ for $\mathcal{L}'$-formulas as follows:
\begin{align*}
  \bot^\circ &:= \bot, \\
  (t_1 = t_2)^\circ &:=
    \begin{cases}
      t_1^\circ = t_2^\circ & \text{if } \mathrm{sort}(t_1^\circ)=\mathrm{sort}(t_2^\circ)\in\mathrm{Sort}, \\
      \top & \text{else;}
    \end{cases} \\
  \bigl(R(t_1,\dots,t_n)\bigr)^\circ &:=
    \begin{cases}
      R(t_1^\circ,\dots,t_n^\circ) & \text{if $R\in\mathrm{Rel}$,} \\
      \top & \text{else;}
    \end{cases} \\
  (\psi\star\chi)^\circ &:= \psi^\circ\star\chi^\circ
    \quad\text{for } \star\in\{\land,\lor,\to\}, \\
  (\exists x^\mi\,\psi)^\circ &:=
    \begin{cases}
      \exists x^\mi\,\psi^\circ & \text{if } \mi\in\mathrm{Sort}, \\
      \exists w^{\mi_0}\psi^\circ              & \text{else;}
    \end{cases} \\
  (\forall x^\mi\,\psi)^\circ &:=
    \begin{cases}
      \forall x^\mi\,\psi^\circ & \text{if } \mi\in\mathrm{Sort}, \\
      \forall w^{\mi_0}\psi^\circ               & \text{else.}
    \end{cases}
\end{align*}

\noindent Now, $\psi^\circ$ is a well-sorted $\mathcal{L}$-formula for every $\mathcal{L}'$-formula $\psi$, and if $\psi$ is an $\mathcal{L}$-formula, then $\psi^\circ = \psi$. Moreover, 
$\bigl(\psi[t^\mi/x^\mi]\bigr)^\circ=\psi^\circ[t^\circ/x]$ whenever $x^\mi$ is distinct from $z_\mi$.

If $\chi$ is an equality axiom of $\mathcal{L}'$, then $\chi^\circ$ is derivable in $\mathcal{L}$. We can now construct the desired derivation by applying $(\cdot)^\circ$ to each step in $\mathcal{D}$.

\end{proof}

\subsection{Liberal equality}\label{libeqsec}

Our notion of equality in $\mathrm{MSL}$ is strict in the sense that $t_1 = t_2$ is well sorted only when $\mathrm{sort}(t_1) = \mathrm{sort}(t_2)$. In some contexts, it is useful to compare terms of different sorts. We introduce \emph{liberal equality}, denoted $\approx$, as an additional binary relation symbol with $\mathrm{rank}(\approx) := \mathrm{Sort}\times \mathrm{Sort}$. We let $\mathcal{L}^\approx$ be the extension of $\mathcal{L}$ with $\approx$ and $\mathrm{Eq}^\approx_\mathcal{L}$ be the set of (the universal closures of) the following:
\begin{itemize}
    \item $x^\mi\approx y^\mi\leftrightarrow x^\mi=y^\mi$
    \item $x^\mi\approx y^\mj\to y^\mj\approx x^\mi$
    \item $x^\mi\approx y^\mj\land y^\mj\approx z^\mk\to x^\mi\approx z^\mk$
    \item $x_1^{\mi_1}\approx y_1^{\mj_1}\land \ldots \land x_n^{\mi_n}\approx y_n^{\mj_n}\rightarrow \left(R(x_1^{\mi_1},\ldots,x_n^{\mi_n})\to R(y_1^{\mj_1},\ldots,y_n^{\mj_n})\right)$, whenever well sorted
    \item $x_1^{\mi_1}\approx y_1^{\mk_1}\land \ldots \land x_n^{\mi_n}\approx y_n^{\mk_n}\rightarrow f(x_1^{\mi_1},\ldots,x_n^{\mi_n})\approx f(y_1^{\mk_1},\ldots, y_n^{\mk_n})$, whenever well sorted
\end{itemize} We write $\Gamma\vdash^{\mathcal{L}}_{\mathrm{MSL}^\approx}\varphi$ as shorthand for $\Gamma,\mathrm{Eq}^\approx_\mathcal{L}\vdash_{\mathrm{MSL}}^{\mathcal{L}^\approx}\varphi.$ We easily get the following.

\begin{proposition}
    $t_1^{\mi_1}\approx t_2^{\mi_2},\,\varphi[t_1^{\mi_1}/x^\mi]\vdash^\mathcal{L}_{\mathrm{MSL}^\approx}\varphi[t_2^{\mi_2}/x^\mi]$, provided $\varphi[t_1^{\mi_1}/x^\mi]$ and $\varphi[t_2^{\mi_2}/x^\mi]$ are well-sorted.
\end{proposition}

\begin{remark} We chose to allow non-strict function symbols in our syntax. This means that we have to be careful when using $\approx$, as an $\mathcal{L}^\approx$-formula can become derivable simply by adding a new function symbol to the language. As an example, take $\varphi$ to be the following formula:
$$\left(\forall x^0\forall y^1(x^0\approx y^1)\land\;\forall z^2\forall w^2(z^2\approx w^2)\right)\to a^0\approx b^2.$$
Clearly, $\not\vdash^\mathcal{L}_{\mathrm{MSL}^\approx}\varphi$ where $\mathcal{L}^\approx$ is the signature of $\varphi$. However, if we extend $\mathcal{L}$ with a function symbol $f$ where $\mathrm{rank}(f)=\{(0,1),(1,2)\}$, then $\vdash^{\mathcal{L}'}_{\mathrm{MSL}^\approx}\varphi$: The first conjunct gives $a\approx f(a)$, and applying $f$ to both sides gives $f(a)\approx f(f(a))$. The second conjunct gives $f(f(a))\approx b$, so $a\approx b$ by transitivity.

\smallskip

This problem does not occur if signatures are required to contain only strict function symbols, as in, e.g., \cite{ManzanoAranda2022}.
\end{remark}

\begin{restatable}{proposition}{libeqext}
\label{LibeqExt}
Suppose that $\mathcal{L}$ and $\mathcal{L}'$ only contain strict function symbols and $\mathcal{L}\leq \mathcal{L}'$. If $\Gamma\cup\{\varphi\}$ is a set of $\mathcal{L}^\approx$-formulas, then
$$\Gamma\vdash_{\mathrm{MSL}^{\approx}}^{\mathcal{L}'}\varphi \quad\text{implies}\quad \Gamma\vdash_{\mathrm{MSL}^{\approx}}^{\mathcal{L}}\varphi.$$
\end{restatable}
\begin{proof}
We place the proof in Appendix~\ref{app:libeq}, since this proposition will not be used elsewhere in the paper.

\end{proof}

\begin{proposition}\label{LibeqConserv}
If $\Gamma\cup \{\varphi\}$ is a set of $\mathcal{L}$-formulas, then
$$\Gamma\vdash_{\mathrm{MSL}}\varphi  \quad\text{iff}\quad \Gamma\vdash^{\mathcal{L}}_{\mathrm{MSL}^\approx}\varphi.$$
\end{proposition}
\begin{proof}
We prove only the right-to-left direction. Let $\mathcal{D}$ be a derivation 
of $\Gamma\vdash^{\mathcal{L}}_{\mathrm{MSL}^\approx}\varphi,$ and replace 
each instance of $t_1^\mi\approx t_2^\mj$ where $\mi\neq \mj$ with $\bot$, and each instance of 
$t_1^\mi\approx t_2^\mi$ with $t_1^\mi = t_2^\mi$. This results in an $\mathcal{L}$-derivation of $\varphi$ from $\Gamma$ together with a set of tautologies.

\end{proof}

\subsection{The standard translation from $\mathrm{MSL}$ to $\mathrm{FOL}$}\label{sec:translation}

Fix a many-sorted signature $\mathcal{L}$. Our target unsorted signature $\mathcal{L}^*$ has the same constant, relation, and function symbols as $\mathcal{L}$, together with a fresh unary relation symbol $Q_\mi$ for each $\mi\in \mathrm{Sort}$. We call these \emph{sort predicates} and write $Q_\mi t$ for $Q_\mi(t)$.

\medskip

The translation $\varphi\mapsto\varphi^*$ is defined as follows:
\begin{enumerate}
    \item $\bot^*:=\bot$;
    \item $(t_1= t_2)^*:=t_1 = t_2$;
    \item $R(t_1,\ldots,t_n)^*:=R(t_1,\ldots,t_n)$;
    \item $(\varphi\star \psi)^*:=\varphi^*\star \psi^*$, for $\star\in \{{\land},{\lor},{\to}\}$;
    \item $(\forall x^\mi \varphi)^*:=\forall x(Q_\mi x\to \varphi^*)$;
    \item $(\exists x^\mi\varphi)^*:=\exists x(Q_\mi x\land \varphi^*)$.
\end{enumerate}
We also extend this to $\mathcal{L}^\approx$ by letting:
\begin{enumerate}
    \item[7.] $(t_1\approx t_2)^*:=t_1=t_2$.
\end{enumerate}

\smallskip

\noindent To accompany this translation, we take $J$ to be (the universal closure of) the following:
\begin{itemize}
    \item $\exists x\, Q_\mi x$ for each $\mi\in \mathrm{Sort}$;
    \item $Q_{\mi_1}x_1\land\ldots\land Q_{\mi_n}x_n\rightarrow Q_\mi f(x_1,\ldots,x_n)$ for each $f\in \mathrm{Fun}$ with $\mathrm{rank}(f)(\mi_1,\dots,\mi_n)=\mi$.
\end{itemize}

\smallskip

\noindent As an instance of the second point, we get:
\begin{itemize}
    \item $Q_\mi c$ for each constant symbol $c$ with sort $\mi$.
\end{itemize}
When we wish to make the underlying many-sorted signature explicit, we write
$J_\mathcal{L}$ instead of $J$.

\smallskip

For well-sorted \emph{sentences} $\Gamma\cup\{\varphi\}$, the correctness of the translation $(\cdot)^*$ can be stated as:
$$\Gamma\vdash_{\mathrm{MSL}}\varphi\quad\text{iff}\quad\Gamma^*,J\vdash_{\mathrm{FOL}}\varphi^*.$$
One direction is well known and relatively straightforward: if $\Gamma\vdash_{\mathrm{MSL}}\varphi$, then one can simulate the many-sorted rules in $\mathrm{FOL}$ to obtain $\Gamma^*, J\vdash_{\mathrm{FOL}}\varphi^*$. We give the details in Appendix~\ref{Forwarddirection}. In the rest of the paper, we will focus on the converse.

\begin{remark} 
Equivalently, we can state the same claim for well-sorted \emph{formulas} as
$$\Gamma\vdash_{\mathrm{MSL}}\varphi\quad\text{iff}\quad \Gamma^*,J,\{Q_\mi x:x^\mi\in \mathrm{FV}(\Gamma,\varphi)\}\vdash_{\mathrm{FOL}}\varphi^*,$$
where $\mathrm{FV}(\Gamma,\varphi)$ is the set of free variables in $\Gamma\cup\{\varphi\}.$ This formulation is equivalent to the original as replacing the free variables $x^\mi\in\mathrm{FV}(\Gamma,\varphi)$ with fresh constants of the corresponding sorts turns $\Gamma\cup\{\varphi\}$ into a set of sentences. However, 
the additional assumptions $Q_\mi x$ in this formulation are necessary. For example,
$$\forall x^\mi P(x^\mi)\vdash_{\mathrm{MSL}}P(x^\mi)\quad\text{but}\quad\forall x(Q_\mi x\to P(x)),J\not\vdash_{\mathrm{FOL}}P(x).$$
\end{remark}

\begin{remark}
By renaming variables, we can assume that any derivation of $\Gamma^*,J\vdash_{\mathrm{FOL}}\varphi^*$ uses only variables from $\mathrm{Var}_\mathcal{L}$.
\end{remark}

\section{Correctness of the standard translation}\label{The general case}

\begin{lemma}\label{LibLemma} Suppose that $\mathcal{L}$ contains only liberal relation and function symbols and finitely many sorts. If $\Gamma\cup \{\varphi\}$ is a set of $\mathcal{L}$-sentences, then
    $$\Gamma\vdash_{\mathrm{MSL}}\varphi\quad\text{iff}\quad\Gamma^*,J\vdash_{\mathrm{FOL}}\varphi^*.$$
\end{lemma}
\begin{proof}
    We only show the right-to-left direction. We recall that $\Gamma\vdash_{\mathrm{MSL}}\varphi$ iff $\Gamma\vdash_{\mathrm{MSL}^\approx}^\mathcal{L}\varphi$, and we define the reverse coding $(\cdot)^\diamond$ as follows:

    \begin{enumerate}
        \item $\bot^\diamond:=\bot$;
        \item $\left(Q_\mi t\right)^\diamond:=\exists x^\mi(t\approx x)$ where $x^\mi$ is a fresh variable; 
        \item $\left(t_1= t_2\right)^\diamond:= t_1\approx t_2 $;
        \item $\left(R(t_1,\dots,t_n)\right)^\diamond:=R(t_1,\dots,t_n)$;
        \item $(A \star B)^\diamond := A^\diamond \star B^\diamond$ for $\star \in \{\land,\lor,\to\}$;
        \item $\left(\exists x A\right)^\diamond:=\bigvee_{\mi\in \mathrm{Sort}}  \exists x_\mi^\mi \left(A^\diamond[x_\mi/x]\right)$ where each  $x_\mi$ is a fresh variable of sort $\mi$;
        \item $\left(\forall x A\right)^\diamond:=\bigwedge_{\mi\in \mathrm{Sort}}  \forall x_\mi^\mi \left(A^\diamond[x_\mi/x]\right)$ where each $x_\mi$ is a fresh variable of sort $\mi$.
    \end{enumerate}
Notice that $\left(\exists x A\right)^\diamond$ and $\left(\forall x A\right)^\diamond$ are well defined because we only have finitely many sorts and all the relation and function symbols are liberal. By simple inductions on the complexity of formulas, we get $$\vdash^\mathcal{L}_{\mathrm{MSL}^\approx}\psi\leftrightarrow \psi^{*\diamond}\quad\text{and}\quad (A[t^\mj/x^\mi])^\diamond= A^\diamond[t^\mj/x^\mi].$$
Next, an induction on the length of derivations gives 
$$\Pi\vdash_\mathrm{FOL} A\quad\text{implies}\quad \Pi^\diamond \vdash^\mathcal{L}_{\mathrm{MSL}^\approx}A^\diamond$$
for all $\mathcal{L}^*$-formulas $\Pi$ and $A$. We include the case for ${\forall} I$ as a demonstration: Assume the last step in the derivation was ${\forall} I$, then $A$ is of the form $\forall xB$ and $\Pi\vdash_\mathrm{FOL}B$ where $x$ is not free in $\Pi$. For each $\mi$, let $x_\mi$ be a fresh variable with sort $\mi$. Since $x_\mi$ is not free in $\Pi$, we get a derivation $\Pi\vdash_\mathrm{FOL}B[x_\mi/x]$ of the same length. By our induction hypothesis, we have $\Pi^\diamond\vdash_\mathrm{MSL}B[x_\mi/x]^\diamond$. So, $\Pi^\diamond\vdash_\mathrm{MSL}\forall x_\mi^\mi B^\diamond[x_\mi/x]$, and therefore $\Pi^\diamond\vdash_\mathrm{MSL}\bigwedge_{\mi\in \mathrm{Sort}}\forall x_\mi^\mi  B^\diamond[x_\mi/x]$.

\smallskip Finally, we notice that each element of $J^\diamond$ is derivable, which gives
\begin{align*}
    \Gamma^*,J\vdash_{\mathrm{FOL}}\varphi^*\quad\text{implies}\quad &\Gamma^{*\diamond},J^\diamond\vdash^\mathcal{L}_{\mathrm{MSL}^\approx}\varphi^{*\diamond}\\
    \text{iff}\quad\quad &\Gamma\vdash^\mathcal{L}_{\mathrm{MSL}^\approx}\varphi.
\end{align*}

\end{proof}

\begin{theorem}\label{mainthrm}
    If $\Gamma\cup \{\varphi\}$ is any set of $\mathcal{L}$-sentences, then 
    $$\Gamma\vdash_{\mathrm{MSL}}\varphi\quad\text{iff}\quad\Gamma^*,J\vdash_{\mathrm{FOL}}\varphi^*.$$
\end{theorem}
\begin{proof}
   We show only the right-to-left direction. Since only finitely many sorts appear in a derivation of $\Gamma^*,J\vdash_{\mathrm{FOL}}\varphi^*,$ we may assume without loss of generality that $\mathrm{Sort}$ is finite. Fix a fresh sort $\mj_0 \notin \mathrm{Sort}$ and let $\mathcal{L}'$ be the signature with $\mathrm{Sort}' := \mathrm{Sort} \cup \{\mj_0\}$ whose relation and function symbols are those of $\mathcal{L}$, all taken to be liberal, with
$$
\mathrm{rank}'(f)(\mi_1,\ldots,\mi_n) := \begin{cases} \mathrm{rank}(f)(\mi_1,\ldots,\mi_n) & \text{if } (\mi_1,\ldots,\mi_n) \in \mathrm{dom}(\mathrm{rank}(f)), \\ \mj_0 & \text{otherwise,} \end{cases}
$$
for each $n$-ary function symbol $f$ and all $(\mi_1,\ldots,\mi_n) \in (\mathrm{Sort}')^n$. Now, an $\mathcal{L}'$-term has a sort other than $\mj_0$ if and only if it is a well-sorted $\mathcal{L}$-term. Clearly, $J\subseteq J'$, where $J'$ are the sort axioms for $\mathcal{L}'$.
So, $\Gamma^*,J'\vdash_{\mathrm{FOL}}\varphi^*$. The previous lemma tells us $\Gamma\vdash^{\mathcal{L}'}_\mathrm{MSL}\varphi$. Thus, $\Gamma\vdash_\mathrm{MSL}^\mathcal{L}\varphi$ by Proposition~\ref{ConLan}.

\end{proof}

\begin{corollary}\label{mainthrmlibeq}
    If $\Gamma\cup \{\varphi\}$ is a set of $\mathcal{L}^\approx$-sentences, then 
    $$\Gamma\vdash^\mathcal{L}_{\mathrm{MSL}^\approx}\varphi\quad\text{iff}\quad\Gamma^*,J\vdash_{\mathrm{FOL}}\varphi^*.$$
\end{corollary}
\begin{proof}
Notice that $\left(\mathrm{Eq}^\approx_\mathcal{L}\right)^*$ is a set of $\mathrm{FOL}$ tautologies. Moreover, $\mathcal{L}^\approx$ only adds a relation symbol to $\mathcal{L}$, so it has the same set of sort axioms. We now have

\begin{align*}
    \Gamma\vdash^\mathcal{L}_{\mathrm{MSL}^\approx}\varphi\quad&\text{iff}\quad\Gamma, \mathrm{Eq}^\approx_\mathcal{L}\vdash_{\mathrm{MSL}}\varphi\\
    &\text{iff}\quad \Gamma^*,J, (\mathrm{Eq^\approx_\mathcal{L}})^*\vdash_{\mathrm{FOL}}\varphi^*\\
    &\text{iff}\quad\Gamma^*,J\vdash_{\mathrm{FOL}}\varphi^*.
\end{align*}
    
\end{proof}

\section{Optional sort axioms}\label{sec:OptAx}
Some authors add additional sort axioms beyond those in $J$. We write $J^+$ for $J$ together with the universal closures of:
\begin{itemize}
    \item $\neg \left(Q_\mi x\land Q_\mj x\right)$ for $\mi\neq \mj$;
    \item $R(x_1,\dots,x_n)\to Q_{\mi_k}x_k$ if $R$ is a strict relation symbol with $\mathrm{rank}(R)=(\mi_1,\ldots,\mi_n)$ and $1\leq k\leq n$.
\end{itemize}
If $\mathrm{Sort}$ is finite, we also include:
\begin{itemize}
    \item $\bigvee_{\mi\in \mathrm{Sort}}Q_\mi x$.
\end{itemize}

\begin{remark}
Again, we must be careful in the presence of liberal equality as, for $\mi\neq\mj$,
$$
J^+\vdash_{\mathrm{FOL}}\left(\forall x^\mi\forall y^\mj\,\neg(x^\mi\approx y^\mj)\right)^*
\quad\text{but}\quad
\not\vdash^{\mathcal{L}}_{\mathrm{MSL}^\approx}\forall x^\mi\forall y^\mj\,\neg(x^\mi\approx y^\mj).
$$
So the additional axioms are not conservative in the presence of liberal equality.
\end{remark}

\begin{theorem}\label{Optaxthrm}
If $\Gamma\cup\{\varphi\}$ is a set of $\mathcal{L}$-sentences, then
$$
\Gamma\vdash^{\mathcal{L}}_{\mathrm{MSL}}\varphi
\quad\text{iff}\quad
\Gamma^*,J^+\vdash_{\mathrm{FOL}}\varphi^*.
$$
\end{theorem}

\begin{proof}
We show only the right-to-left direction and assume that
$
\Gamma^*,J^+\vdash_{\mathrm{FOL}}\varphi^*.
$
Since only finitely many sorts can occur in a derivation, we may assume without loss of generality that $\mathcal{L}$ has only finitely many sorts. Fix an extension $\mathcal{L}'$ of $\mathcal{L}$ with the same sorts, and in which all relation and function symbols are liberal.\footnote{To see that such an extension exists, we can fix some $\mi_0\in \mathrm{Sort}$ and extend $\mathrm{rank}(f)$ by assigning $\mi_0$ to all previously undefined inputs.}

We take
$$
\Sigma_1:=\{\forall x^\mi\forall y^\mj\,\neg(x^\mi\approx y^\mj): \mi\neq \mj\},
$$
and let $\Sigma_2$ be the set containing
$$
\forall x_1^{\mj_1}\cdots \forall x_n^{\mj_n}
\left(
R(x_1^{\mj_1},\ldots,x_n^{\mj_n})
\to
\exists y^{\mi_k}(x_k^{\mj_k}\approx y^{\mi_k})
\right)
$$
for each strict $n$-ary relation symbol $R$ of $\mathcal{L}$ with $\mathrm{rank}(R)=(\mi_1,\ldots,\mi_n)$, each $(\mj_1,\ldots,\mj_n)\in \mathrm{Sort}^n$, and each $k\in\{1,\ldots,n\}$. 

\begin{claim}
$\Gamma,\Sigma_1,\Sigma_2\vdash^{\mathcal{L}'}_{\mathrm{MSL}^\approx}\varphi$.
\end{claim}

\begin{claimproof}[Proof of claim]
From  $\Gamma^*,J^+\vdash_{\mathrm{FOL}}\varphi^*$ we get 
$\Gamma^*,J^+,J_{\mathcal{L}'}\vdash_{\mathrm{FOL}}\varphi^*$. Now, let $(\cdot)^\diamond$ be the reverse coding from the proof of Lemma~\ref{LibLemma}. 
Since $\mathcal{L}'$ has only liberal relation and function symbols and finitely many sorts, we get
$
\Gamma,(J^+)^\diamond\vdash^{\mathcal{L}'}_{\mathrm{MSL}^\approx}\varphi.
$
Each element of $(J^+)^\diamond$ is derivable from $\Sigma_1\cup\Sigma_2$. Therefore,
$
\Gamma,\Sigma_1,\Sigma_2\vdash^{\mathcal{L}'}_{\mathrm{MSL}^\approx}\varphi.
$

\end{claimproof}

\begin{claim}\label{claimtwo}
$\Gamma,\Sigma_1\vdash^{\mathcal{L}'}_{\mathrm{MSL}^\approx}\varphi$.
\end{claim}

\begin{claimproof}[Proof of claim]
For each strict $R$ from $\mathcal{L}$ and terms $t_1,\dots,t_n$ of any sorts, we let
$$
\zeta_R(t_1,\dots,t_n):=\exists y_1^{\mi_1}\ldots \exists y_n^{\mi_n}\left(t_1\approx y_1\land \ldots \land t_n\approx y_n\land R(y_1,\ldots, y_n)\right)
$$
where $\mathrm{rank}(R)=(\mi_1,\dots, \mi_n)$ and $y_1^{\mi_1},\ldots,y_n^{\mi_n}$ are fresh and pairwise distinct. We define $(\cdot)^\zeta$ on $(\mathcal{L}')^\approx$-formulas by replacing each $R(t_1,\dots, t_n)$ with $\zeta_R(t_1,\dots, t_n)$ for strict $R$.

Now $\vdash^{\mathcal{L}'}_{\mathrm{MSL}^\approx}\psi^\zeta\leftrightarrow \psi$ for every $\mathcal{L}$-formula, and $(\psi[t^\mi/x^\mi])^\zeta=\psi^\zeta[t^\mi/x^\mi]$ for every $\mathcal{L}'$-formula. Moreover, $\Sigma_1^\zeta=\Sigma_1$, and $\vdash^{\mathcal{L}'}_{\mathrm{MSL}^\approx}\chi^\zeta$ for every $\chi\in \Sigma_2$.

By applying $(\cdot)^\zeta$ throughout a derivation of $\Gamma,\Sigma_1,\Sigma_2\vdash^{\mathcal{L}'}_{\mathrm{MSL}^\approx}\varphi$, we get $\Gamma^\zeta,\Sigma_1^\zeta,\Sigma_2^\zeta\vdash^{\mathcal{L}'}_{\mathrm{MSL}^\approx}\varphi^\zeta$. Thus,
${\Gamma,\Sigma_1\vdash^{\mathcal{L}'}_{\mathrm{MSL}^\approx}\varphi.}$
\end{claimproof}

\begin{claim}
$\Gamma\vdash^{\mathcal{L}'}_{\mathrm{MSL}}\varphi$.
\end{claim}

\begin{claimproof}[Proof of claim]
Define $(\cdot)^\circ$ on $(\mathcal{L}')^\approx$-formulas by replacing each instance of $t_1^\mi\approx t_2^\mj$ with $t_1=t_2$ if $\mi=\mj$, and with $\bot$ if $\mi\neq \mj$. Then $\psi^\circ=\psi$ for every $\mathcal{L}'$-formula $\psi$, and both $\Sigma_1^\circ$ and $\left(\mathrm{Eq}_{\mathcal{L}'}^\approx\right)^\circ$ are sets of formulas derivable in $\mathrm{MSL}^{\mathcal{L}'}$. A simple induction on the length of derivations shows that
$$
\Sigma\vdash^{\mathcal{L}'}_{\mathrm{MSL}^\approx}\psi
\quad\text{implies}\quad
\Sigma^\circ\vdash^{\mathcal{L}'}_{\mathrm{MSL}}\psi^\circ.
$$
Applying this to Claim~\ref{claimtwo}, we get
$
\Gamma\vdash^{\mathcal{L}'}_{\mathrm{MSL}}\varphi.
$

\end{claimproof}

Finally,  Proposition~\ref{ConLan} gives
$
\Gamma\vdash^{\mathcal{L}}_{\mathrm{MSL}}\varphi
$
since $\Gamma\cup\{\varphi\}$ is a set of $\mathcal{L}$-sentences and $\mathcal{L}\leq\mathcal{L}'$.

\end{proof}

\section{Second-order logic}\label{Second-order logic}

In this section, we show that van Dalen's translation from second-order logic to unsorted first-order logic goes through. Our presentation differs slightly from van Dalen's but is easily seen to be equivalent. 

\begin{definition}\label{def:SOL}
The logical signs of second-order logic ($\mathrm{SOL}$) are those of
$\mathrm{FOL}$ together with a countably infinite set
of \emph{$n$-ary predicate variables}
$\mathrm{Var}_n$ for each $n\geq 0$. These are taken to be pairwise disjoint and disjoint
from $\mathrm{Var}$. We use uppercase letters for the variables of $\mathrm{Var}_n$ and write $X^n$ to indicate $X\in \mathrm{Var}_n$. Similarly, we write $R^n$ to indicate that the relation symbol $R$ has arity $n$. A \emph{second-order term} of arity $n$ is an $n$-ary predicate variable or an $n$-ary relation symbol from $L$. We count $\bot$ as a second-order term of arity $0$.

\smallskip

Given an unsorted signature $L$, the $L$-terms of $\mathrm{SOL}$ are the $L$-terms of $\mathrm{FOL}$. The \emph{$L$-formulas} of $\mathrm{SOL}$  are defined inductively as follows:
\begin{enumerate}
    \item $\bot$ is a formula.
    \item If $t_1,\ldots,t_n$ are first-order terms, then $R^n(t_1,\ldots,t_n)$ and $X^n(t_1,\ldots,t_n)$ are formulas.
    \item If $\varphi$ and $\psi$ are formulas, then so is $\varphi\star\psi$ for $\star\in\{\land,\lor,\to\}$.
    \item If $\varphi$ is a formula, then so are $\forall x\,\varphi$, $\exists x\,\varphi$, $\forall X^n\,\varphi$, and $\exists X^n\,\varphi$.
\end{enumerate}
Note that equality is not included as a primitive. If $T^n$ is a second-order term of arity $n$, we write $\varphi[T^n/X^n]$
for the result of replacing every free instance of $X^n$ with $T^n$. 
\end{definition}

We write $\Gamma\vdash_{\mathrm{SOL}}\varphi$ to indicate that $\varphi$ is
derivable from $\Gamma$ using the rules from Table~\ref{FOLrules} together
with the second-order quantifier rules from Table~\ref{SOLrules}, 
and (the universal closure of) the following \emph{comprehension schema}:
$$
\exists X^n\,\forall x_1\ldots\forall x_n
\big(X^n(x_1,\ldots,x_n)\leftrightarrow \psi\big),
$$
where $X^n$ does not occur free in $\psi$.

\begin{table}[H]
\centering
\caption{Second-order quantifier rules for $\mathrm{SOL}$}
\label{SOLrules}
\vspace{1em}
\footnotesize\begin{tabular}{ccc}

\begin{minipage}{0.26\textwidth}
\begin{prooftree}
\AxiomC{$\varphi$}
\RightLabel{${\forall}^2 I$}
\UnaryInfC{$\forall X^n\, \varphi$}
\end{prooftree}
\end{minipage} &
\begin{minipage}{0.30\textwidth}
\begin{prooftree}
\AxiomC{$\forall X^n\, \varphi$}
\RightLabel{${\forall}^2 E$}
\UnaryInfC{$\varphi[T^n/X^n]$}
\end{prooftree}
\end{minipage} &
\begin{minipage}{0.26\textwidth}
\end{minipage} \\[3em]

\begin{minipage}{0.26\textwidth}
\begin{prooftree}
\AxiomC{$\varphi[T^n/X^n]$}
\RightLabel{${\exists}^2 I$}
\UnaryInfC{$\exists X^n\, \varphi$}
\end{prooftree}
\end{minipage} &
\begin{minipage}{0.30\textwidth}
\begin{prooftree}
\AxiomC{$\exists X^n\, \varphi$}
\AxiomC{$[\varphi]$}
\noLine
\UnaryInfC{$\vdots$}
\noLine
\UnaryInfC{$\chi$}
\RightLabel{${\exists}^2 E$}
\BinaryInfC{$\chi$}
\end{prooftree}
\end{minipage} &
\begin{minipage}{0.26\textwidth}
\end{minipage} \\[2em]

\end{tabular}

\medskip

\begin{minipage}{0.8\textwidth}
\footnotesize
In ${\forall}^2 I$ and ${\exists}^2 E$, the predicate
variable $X^n$ may not occur free in any undischarged assumption. In $\exists^2 E$, $X^n$ may also not occur free in $C$.
\end{minipage}
\end{table}

For simplicity, we will translate second-order logic into first-order logic
with $\mathrm{Var}\cup\bigcup_{n\geq 0}\mathrm{Var}_n$ as its pool of
variables.

\smallskip

Given a signature $L$, the target signature $L^\dagger$
has:
\begin{itemize}
    \item all function symbols from $L$;
    \item a fresh constant $c_R$ for each relation symbol $R$ from $L$;
    \item a fresh constant $c_\bot$;
    \item an $(n+1)$-ary relation symbol $Ap_n$ for each $n\geq 0$;
    \item unary relation symbols $V$ and $U_n$ for each $n\geq 0$.
\end{itemize}

\noindent The translation $\varphi\mapsto\varphi^\dagger$ of $\mathrm{SOL}$-formulas
over $L$ into $L^\dagger$-formulas is defined as follows:
\begin{enumerate}
    \item $\bot^\dagger := Ap_0(c_\bot)$;
    \item $(X^n(t_1,\ldots,t_n))^\dagger := Ap_n(X,t_1,\ldots,t_n)$;
    \item $(R(t_1,\ldots,t_n))^\dagger := Ap_n(c_R,t_1,\ldots,t_n)$;
    \item $(\varphi\star\psi)^\dagger := \varphi^\dagger\star\psi^\dagger$
          for $\star\in\{\land,\lor,\to\}$;
    \item $(\forall x\varphi)^\dagger
          := \forall x(V(x)\to\varphi^\dagger)$
          and $(\exists x\varphi)^\dagger
          := \exists x(V(x)\land\varphi^\dagger)$;
    \item $(\forall X^n\,\varphi)^\dagger
          := \forall X(U_n(X)\to\varphi^\dagger)$
          and $(\exists X^n\,\varphi)^\dagger
          := \exists X(U_n(X)\land\varphi^\dagger)$.
\end{enumerate}

\noindent We let $G$ be (the universal closure of) the following axioms:
\begin{itemize}
    \item $\exists x\,V(x)$ and $\exists x\,U_n(x)$;
    \item $\neg\left(V(x) \land U_n(x)\right)$; 
    \item $\neg\left(U_n(x) \land U_m(x)\right)$ for $m\neq n$;
    \item $Ap_n(x,y_1,\ldots,y_n)\to U_n(x)\land
          \bigwedge_{k=1}^n V(y_k)$;
    \item $V(c)$ for each constant $c$ from $L$;
    \item $U_n(c_R)$ for each $n$-ary relation symbol $R$ from $L$;
    \item $U_0(c_\bot)$ and $\neg Ap_0(c_\bot)$;
    \item $V(x_1)\land\ldots\land V(x_n)\to V(f(x_1,\ldots,x_n))$
          for each $n$-ary function symbol $f$ from $L$;
    \item $\psi^\dagger$ for each instance of the comprehension schema.
\end{itemize}

\begin{theorem}\label{SOLtoFOL}
If $\Gamma\cup\{\varphi\}$ is a set of $\mathrm{SOL}$-sentences of $L$,
then
$$
\Gamma\vdash_{\mathrm{SOL}}\varphi
\quad\text{if and only if}\quad
\Gamma^\dagger, G\vdash_{\mathrm{FOL}}\varphi^\dagger.
$$
\end{theorem}
\begin{proof}
We decompose $(\cdot)^\dagger$ into two steps by first giving a translation $(\cdot)^\bullet$ into $\mathrm{MSL}$ . We define the many-sorted signature $\mathcal{L}^\bullet$ with the sorts $\{\frak{i}\}\cup \mathbb{N}$ where $\frak{i}$ is the sort for individuals and $n\in \mathbb{N}$ is the sort for $n$-ary predicates. We take the first-order variables as the variables of sort $\mathfrak{i}$ and the $n$-ary predicate variables as the variables of sort $n$.

We include the function symbols from $L$ and assign each $n$-ary function symbol the rank $(\frak{i},\dots,\frak{i},\frak{i})$. We include the constant symbols $c_R$ with sort $n$ for each $n$-ary relation symbol $R$, and the constant symbol $c_\bot$ with sort 0. Finally, for each $n$, we include the $(n+1)$-ary relation symbol $Ap_n$ with $\mathrm{rank}(Ap_n)=(n,\frak{i},\dots,\frak{i} ).$ We now define $(\cdot)^\bullet$ from $\mathrm{SOL}$-formulas to $\mathcal{L}^\bullet$-formulas by:
\begin{enumerate}
    \item $\bot^\bullet:=Ap_0(c_\bot)$;
    \item $(X^n(t_1,\ldots,t_n))^\bullet:=Ap_n(X^n,t_1,\ldots,t_n)$;
    \item $(R(t_1,\ldots,t_n))^\bullet:=Ap_n(c_R,t_1,\ldots,t_n)$;
    \item $(\varphi\star\psi)^\bullet:=\varphi^\bullet\star\psi^\bullet$ for $\star\in\{\land,\lor,\to\}$;
    \item $(\forall x\varphi)^\bullet:=\forall x\varphi^\bullet$ and $(\exists x\varphi)^\bullet:=\exists x\varphi^\bullet$;
    \item $(\forall X^n\,\varphi)^\bullet:=\forall X^n\,\varphi^\bullet$ and $(\exists X^n\,\varphi)^\bullet:=\exists X^n\,\varphi^\bullet$.
\end{enumerate}
We take $H$ to be the set of $\mathcal{L}^\bullet$-sentences consisting of $\neg Ap_0(c_\bot)$ together with $\psi^\bullet$ for each instance of the comprehension schema. 

\begin{claim}\label{SOLtoMSL}
$\Gamma\vdash_{\mathrm{SOL}}\varphi$ if and only if $\Gamma^\bullet, H\vdash_{\mathrm{MSL}}\varphi^\bullet.$
\end{claim}

\begin{claimproof}[Proof of Claim] We only show the right-to-left direction. The only sticking point is that $\mathrm{SOL}$ does not include equality as a primitive, while $\mathrm{MSL}$ does. So on the $\mathrm{SOL}$ side, we use \emph{Leibniz equality} for individuals:
$$t_1\sim t_2\;:=\;\forall X^1\big(X(t_1)\leftrightarrow X(t_2)\big),$$
and for predicates, we use \emph{extensionality}:
$$T_1^n\sim T_2^n\;:=\;\forall x_1\ldots\forall x_n\big(T_1^n(x_1,\ldots,x_n)\leftrightarrow T_2^n(x_1,\ldots,x_n)\big).$$
We can now transform any derivation of $\Gamma^\bullet, H\vdash_{\mathrm{MSL}}\varphi^\bullet$ by replacing each $Ap_n(c_R,t_1,\ldots,t_n)$ with $R^n(t_1,\ldots,t_n)$, each $Ap_n(X^n,t_1,\ldots,t_n)$ with $X^n(t_1,\ldots,t_n)$, each $Ap_0(c_\bot)$ with $\bot$, and each instance of $=$ with $\sim $. Then, $\psi^\bullet$ is mapped back to $\psi$ for each second-order formula $\psi$. Moreover, it is routine to check that the image of each equality axiom is derivable in $\mathrm{SOL}$. 
    
\end{claimproof}

Finally, we identify $V$ with $Q_{\mathfrak{i}}$ and $U_n$ with $Q_n$. This gives $(\psi^\bullet)^*=\psi^\dagger$ for each second-order $\psi$ and $H^*\cup J^+_{\mathcal{L}^\bullet}=G$. So by Theorem~\ref{Optaxthrm},
$$
\Gamma^\bullet, H\vdash_{\mathrm{MSL}}^{\mathcal{L}^\bullet}\varphi^\bullet
\quad\text{iff}\quad
\Gamma^\dagger, G\vdash_{\mathrm{FOL}}\varphi^\dagger.
$$

\end{proof}

\section{Simple proof of Herbrand's claim}\label{Repairing Herbrand's proof}
In this section, we only consider equality-free formulas with strict relation and function symbols. Accordingly, we write $\mathrm{rank}(R)=(\mi_1,...,\mi_n)$ and $\mathrm{rank}(f)=(\mi_1,...,\mi_n,\mi)$ instead of $(\mi_1,...,\mi_n)\in \mathrm{rank}(R)$ and $(\mi_1,...,\mi_n,\mi)\in \mathrm{rank}(f)$ for $n$-ary relation and function symbols, respectively. We will also use without mention that any derivation of $\Pi\vdash_\mathrm{FOL}A$ can be effectively transformed into an equality-free one (see, e.g., \cite{Takeuti87}, where this is established via cut-elimination). 

\smallskip

As mentioned in the introduction, this is the setting Herbrand was working in. He gave the following argument of how one gets $\Gamma\vdash_\mathrm{MSL}\varphi$ from $\Gamma^*,J\vdash_{\mathrm{FOL}}\varphi^*$: Take a derivation of $\Gamma^*,J\vdash_{\mathrm{FOL}}\varphi^*$ and replace every occurrence of $Q_\mi t$ by $\top$. This transforms $\varphi^*$ into a sentence equivalent to $\varphi$, $\Gamma^*$ into sentences equivalent to $\Gamma$, and each axiom in $J$ into a tautology. This gives $\Gamma \vdash_\mathrm{FOL} \varphi$. However, as Schmidt pointed out, it is not obvious that $\Gamma \vdash_\mathrm{MSL} \varphi$, since the derivation may pass through formulas that are not well-sorted. 

However, we can use insights from \cite{NourRaffalli2003} to obtain a many-sorted derivation from an unsorted one. First, if we ignore function symbols for a moment and consider a formula like $R(x,x)$ where $\mathrm{rank}(R)=(\mi,\mj)$ with $\mi\neq \mj$. Then we can imagine the placements of the instances of $x$ telling us what role it should play: the first instance should be understood as a variable with sort $\mi$, while the second should be understood as having $\mj$ as its sort. We can make this precise by introducing variables $x^{\langle \mi\rangle}$ and $x^{\langle \mj\rangle}$ with the sorts $\mi$ and $\mj$, respectively. Now, the ill-sorted $R(x,x)$ becomes the well-sorted $R(x^{\langle \mi\rangle},x^{\langle \mj\rangle})$. Similarly, $\exists x R(x,x)$ becomes $\exists x^{\langle \mi\rangle}\exists x^{\langle \mj\rangle} R(x^{\langle \mi\rangle},x^{\langle \mj\rangle})$. 

The tricky part is dealing with complex terms, as in $R(f(x), f(x))$ where $\mathrm{rank}(f)=(\mk,\ml)$. Here, the formula is telling us that the first instance of the term $f(x)$ should have the sort $\mi$, while the second should have the sort $\mj$. This time we introduce the function symbols $f^{\langle \mi\rangle}$ and $f^{\langle \mj\rangle}$ with $\mathrm{rank}(f^{\langle \mi\rangle})=(\mk,\mi)$ and $\mathrm{rank}(f^{\langle \mj\rangle})=(\mk,\mj)$, respectively. Then $R(f(x), f(x))$ becomes $R(f^{\langle \mi\rangle}(x^{\langle \mk\rangle}), f^{\langle \mj\rangle}(x^{\langle \mk\rangle}))$ as both $f^{\langle \mi\rangle}$ and $f^{\langle \mj\rangle}$ expect inputs with sort $\mk$. Whenever a variable or function symbol already has the required sort or rank, we can leave it unchanged. 

\begin{theorem}\label{HebLem}
    If $\Gamma\cup\{\varphi\}$ is a set of equality-free formulas with only strict symbols, then
    $$\Gamma\vdash_{\mathrm{MSL}}\varphi  \quad\text{iff}\quad\Gamma\vdash_{\mathrm{FOL}}\varphi.$$
\end{theorem}
\begin{proof} Fix a derivation $\mathcal{D}$ of $\Gamma\vdash_\mathrm{FOL}\varphi$. For each variable $x^\mi$ appearing in $\mathcal{D}$ and each sort $\mj \neq \mi$, we pick a fresh variable $x^{\langle \mj\rangle}\in\mathrm{Var}_\mj$, and we let $x^{\langle \mi\rangle} := x$. If $\mathrm{rank}(f)=(\mi_1,\ldots,\mi_n,\mi)$ and $\mj \neq \mi$, then we introduce a new function symbol $f^{\langle \mj\rangle}$ with $\mathrm{rank}(f^{\langle \mj\rangle})=(\mi_1,\ldots,\mi_n,\mj)$, and let $f^{\langle \mi\rangle} := f$. 
We extend this to complex terms by letting $$(f(t_1,\ldots,t_n))^{\langle \mj\rangle} \;:=\; f^{\langle \mj\rangle}(t_1^{\langle \mi_1\rangle},\ldots,t_n^{\langle \mi_n\rangle}).$$  We take $\mathcal{L}'$ to be the signature extended with the new function symbols. Notice that $\mathrm{sort}(t^{\langle \mi\rangle})=\mi$ for any term $t$ and $\mi\in\mathrm{Sort}$. 

We recursively define $A^\diamond$ as follows:
\begin{enumerate}
    \item $\bot^\diamond := \bot$;
    \item $(R(t_1,\ldots,t_n))^\diamond := R(t_1^{\langle \mi_1\rangle},\ldots,t_n^{\langle \mi_n\rangle})$ where $\mathrm{rank}(R) = (\mi_1,\ldots,\mi_n)$;
    \item $(A \star B)^\diamond := A^\diamond \star B^\diamond$ for $\star \in \{\land,\lor,\to\}$;
    \item $(\exists x A)^\diamond := \exists x^{\langle \mi_1\rangle} \cdots\exists x^{\langle \mi_m\rangle} A^\diamond$ where $\mi_1=\mathrm{sort}(x)$ and $\mi_2,\ldots, \mi_m$ are the sorts $\mj\neq \mi_1$ for which $x^{\langle \mj\rangle}\in \mathrm{FV}(A^\diamond)$;
    \item $(\forall x A)^\diamond := \forall x^{\langle \mi_1\rangle}\cdots\forall x^{\langle \mi_m\rangle} A^\diamond$ where $\mi_1=\mathrm{sort}(x)$ and $\mi_2,\ldots, \mi_m$ are the sorts $\mj\neq \mi_1$ for which $x^{\langle \mj\rangle}\in \mathrm{FV}(A^\diamond)$.
\end{enumerate}
Clearly, $\psi^\diamond = \psi$ for all $\mathcal{L}$-formulas. We will also need the following claim. Its proof is a straightforward but lengthy induction placed in  Appendix~\ref{app:FormSubst}.

\begin{restatable}{claim}{FormSubst}\label{FormSubst}
    Let $A$ be a formula, $t$ a term free for $x$ in $A$, and suppose every free $x^{\langle \mi\rangle}$ in $A^\diamond$ is among $x^{\langle \mi_1\rangle}, \ldots, x^{\langle \mi_m\rangle}$. Then each $t^{\langle \mi_k\rangle}$ is free for $x^{\langle \mi_k\rangle}$ in $A^\diamond$ and
    $(A[t/x])^\diamond = A^\diamond[t^{\langle \mi_1\rangle}/x^{\langle \mi_1\rangle}, \ldots, t^{\langle \mi_m\rangle}/x^{\langle \mi_m\rangle}].$
\end{restatable}

We can now get a derivation of $\Gamma\vdash^{\mathcal{L}'}_\mathrm{MSL}\varphi$ from $\mathcal{D}$ by applying $(\cdot)^\diamond$ to each formula and replacing each quantifier rule with one of the following derivable rules:

\begin{center}
\footnotesize
\begin{tabular}{cc}

\begin{minipage}{0.40\textwidth}
\begin{prooftree}
\AxiomC{$\psi$}
\RightLabel{$\forall I_n$}
\UnaryInfC{$\forall x_1^{\mi_1}\cdots\forall x_n^{\mi_n}\, \psi$}
\end{prooftree}
\end{minipage} &
\begin{minipage}{0.40\textwidth}
\begin{prooftree}
\AxiomC{$\forall x_1^{\mi_1}\cdots\forall x_n^{\mi_n}\, \psi$}
\RightLabel{$\forall E_n$}
\UnaryInfC{$\psi[t_1^{\mi_1}/x_1^{\mi_1},\ldots,t_n^{\mi_n}/x_n^{\mi_n}]$}
\end{prooftree}
\end{minipage} \\[3em]

\begin{minipage}{0.40\textwidth}
\begin{prooftree}
\AxiomC{$\psi[t_1^{\mi_1}/x_1^{\mi_1},\ldots,t_n^{\mi_n}/x_n^{\mi_n}]$}
\RightLabel{$\exists I_n$}
\UnaryInfC{$\exists x_1^{\mi_1}\cdots\exists x_n^{\mi_n}\, \psi$}
\end{prooftree}
\end{minipage} &
\begin{minipage}{0.40\textwidth}
\begin{prooftree}
\AxiomC{$\exists x_1^{\mi_1}\cdots\exists x_n^{\mi_n}\, \psi$}
\AxiomC{$[\psi]$}
\noLine
\UnaryInfC{$\vdots$}
\noLine
\UnaryInfC{$\chi$}
\RightLabel{$\exists E_n$}
\BinaryInfC{$\chi$}
\end{prooftree}
\end{minipage} \\[2em]

\end{tabular}
\end{center}

By Proposition \ref{ConLan}, we now get $\Gamma\vdash_\mathrm{MSL}\varphi$.

\end{proof}

\bibliographystyle{alpha}
\bibliography{references}

@phdthesis{THESE_1930__110__1_0,
     author = {Herbrand, Jacques},
     title = {Recherches sur la th\'eorie de la d\'emonstration},
     series = {Th\`eses de l'entre-deux-guerres},
     year = {1930},
     number = {110},
     language = {fr},
     url = {https://www.numdam.org/item/THESE_1930__110__1_0/}
}

@article{Schmidt1938,
  author    = {Schmidt, Arnold},
  title     = {Über deduktive {T}heorien mit mehreren {S}orten von {G}runddingen},
  journal   = {Mathematische Annalen},
  volume    = {115},
  number    = {1},
  pages     = {485--506},
  year      = {1938},
  doi       = {10.1007/BF01448954},
}

@article{Schmidt1951,
  author    = {Schmidt, Arnold},
  title     = {Die {Z}ul\"assigkeit der {B}ehandlung mehrsortiger {T}heorien mittels der \"ublichen einsortigen {P}r\"adikatenlogik},
  journal   = {Mathematische Annalen},
  volume    = {123},
  number    = {1},
  pages     = {187--200},
  year      = {1951},
  doi       = {10.1007/BF02054948},
}

@article{Wang1952,
  author    = {Wang, Hao},
  title     = {Logic of Many-Sorted Theories},
  journal   = {Journal of Symbolic Logic},
  volume    = {17},
  number    = {2},
  pages     = {105--116},
  year      = {1952},
  doi       = {10.2307/2266241},
}

@article{Gilmore1958,
     author = {Gilmore, Paul. C.},
     title = {An addition to {\textquotedblleft}logic of many-sorted theories{\textquotedblright}},
     journal = {Compositio Mathematica},
     pages = {277--281},
     year = {1958},
     publisher = {Kraus Reprint},
     volume = {13},
     mrnumber = {106176},
     zbl = {0149.24502},
     language = {en},
     url = {https://www.numdam.org/item/CM_1956-1958__13__277_0/}
}

@article{NourRaffalli2003,
  author    = {Nour, Karim and Raffalli, Christophe},
  title     = {Simple Proof of the Completeness Theorem for Second-Order Classical and Intuitionistic Logic by Reduction to First-Order Mono-Sorted Logic},
  journal   = {Theoretical Computer Science},
  volume    = {308},
  number    = {1},
  pages     = {227--237},
  year      = {2003},
  doi       = {10.1016/S0304-3975(02)00731-4},
}

@incollection{Herbrand1971,
  author    = {Herbrand, Jacques},
  title     = {Investigations in Proof Theory},
  booktitle = {Logical Writings},
  editor    = {Goldfarb, Warren D.},
  publisher = {D. Reidel},
  address   = {Dordrecht},
  pages     = {44--202},
  year      = {1971},
  doi       = {10.1007/978-94-010-3072-4},
  isbn      = {978-90-277-0176-3},
}

@incollection{ManzanoAranda2022,
  author    = {Manzano, Mar{\'\i}a and Aranda, V{\'\i}ctor},
  title     = {Many-Sorted Logic},
  booktitle = {The {Stanford} Encyclopedia of Philosophy},
  editor    = {Zalta, Edward N. and Nodelman, Uri},
  year      = {2022},
  edition   = {Winter 2022},
  publisher = {Metaphysics Research Lab, Stanford University},
  url       = {https://plato.stanford.edu/archives/win2022/entries/logic-many-sorted/},
}

@book{vanDalen1980,
  author    = {van Dalen, Dirk},
  title     = {Logic and Structure},
  series    = {Universitext},
  publisher = {Springer-Verlag},
  address   = {Berlin},
  year      = {1980},
  isbn      = {3-540-09893-3}
}

@book{vanDalen1994,
  author    = {van Dalen, Dirk},
  title     = {Logic and Structure},
  edition   = {3rd},
  series    = {Universitext},
  publisher = {Springer-Verlag},
  address   = {Berlin},
  year      = {1994},
  isbn      = {978-3-540-57839-0}
}

@book{Takeuti87,
   title =     {Proof Theory},
   author =    {Gaisi Takeuti},
   publisher = {North-Holland},
   isbn =      {0444104925; 9780444104922; 0720422000; 9780720422009; 0720422779; 9780720422771},
   year =      {1987},
   series =    {Studies in Logic and the Foundations of Mathematics 81},
   edition =   {2nd},
   }

@article{GOGUEN1992217,
title = {Order-sorted algebra I: equational deduction for multiple inheritance, overloading, exceptions and partial operations},
journal = {Theoretical Computer Science},
volume = {105},
number = {2},
pages = {217-273},
year = {1992},
issn = {0304-3975},
doi = {https://doi.org/10.1016/0304-3975(92)90302-V},
url = {https://www.sciencedirect.com/science/article/pii/030439759290302V},
author = {Joseph A. Goguen and José Meseguer}
}

\appendix 

\section{The forward direction}\label{Forwarddirection}

In this appendix, we prove the left-to-right direction of the standard translation. We will use without further mention that, whenever $t^\mi$ is free for $x^\mi$ in $\psi$,
$$
(\psi[t^\mi/x^\mi])^*=\psi^*[t/x].
$$
Moreover, if $t^\mi$ is an $\mathcal{L}$-term, then
$$
J,\{Q_\mi x^\mi:x^\mi\in\mathrm{FV}(t)\}\vdash_{\mathrm{FOL}}Q_\mi t.
$$

\begin{lemma}\label{RemoveQAssumptions}
Let $\Gamma\cup\{\varphi\}$ be a set of $\mathcal{L}$-formulas, and let $S\subseteq \mathrm{Var}_{\mathcal{L}}$. Then
$$
\Gamma^*,J,\{Q_\mi x^\mi:x^\mi\in S\}\vdash_{\mathrm{FOL}}\varphi^*
\quad\text{implies}\quad
\Gamma^*,J,\{Q_\mi x^\mi:x^\mi\in \mathrm{FV}(\Gamma,\varphi)\}\vdash_{\mathrm{FOL}}\varphi^*.
$$
\end{lemma}

\begin{proof} By weakening, we may assume that $\mathrm{FV}(\Gamma,\varphi)\subseteq S$. 
Fix a derivation $\mathcal{D}$ of
$
\Gamma^*,J,\{Q_\mi x^\mi:x^\mi\in S\}\vdash_{\mathrm{FOL}}\varphi^*,
$ and let $S_0\subseteq S$ be the finite set of assumptions from $S$ used in $\mathcal{D}$. 

For $y^\mj\in S_0\setminus \mathrm{FV}(\Gamma,\varphi)$, we have $J\vdash \exists yQ_\mj y$, so we can apply $\exists E$ to get a derivation of 
$$
\Gamma^*,J,\{Q_\mi x^\mi:x^\mi\in S_0\}\setminus\{Q_\mj y^\mj\}\vdash_{\mathrm{FOL}}\varphi^*.
$$
We can repeat this process until we get a derivation of $\Gamma^*,J,\{Q_\mi x^\mi:x^\mi\in \mathrm{FV}(\Gamma,\varphi)\}\vdash_{\mathrm{FOL}}\varphi^*.$

\end{proof}

\begin{theorem}\label{ForwardDirectionFormulas}
If $\Gamma\cup\{\varphi\}$ is a set of $\mathcal{L}$-formulas, then
$$
\Gamma\vdash_{\mathrm{MSL}}^{\mathcal{L}}\varphi\quad\text{implies}\quad\Gamma^*,J,\{Q_\mi x^\mi:x^\mi\in \mathrm{FV}(\Gamma,\varphi)\}\vdash_{\mathrm{FOL}}\varphi^*.
$$
\end{theorem}

\begin{proof}
Fix a derivation $\mathcal{D}$ of $\Gamma\vdash_\mathrm{MSL}\varphi$. We take as our induction hypothesis that for any set of $\mathcal{L}$-formulas $\Delta\cup\{\psi\}$, if there is a derivation of $\Delta\vdash_\mathrm{MSL}\psi$  of length strictly less than $\mathcal{D}$, then $\Delta^*,J,\{Q_\mi x^\mi:x^\mi\in \mathrm{FV}(\Delta,\psi)\}\vdash_{\mathrm{FOL}}\psi^*.$

We proceed by cases on what the last rule applied in $\mathcal{D}$ was. The cases for assumptions, equality axioms, and propositional rules are immediate. The quantifier rules then follow from the lemma above. We spell out the cases for $\exists I$ and $\exists E$ as the cases for $\forall I$ and $\forall E$ are similar.

\medskip

\noindent\textbf{Case $\exists I$:} 
Suppose that the last rule was $\exists I$. Then $\varphi$ is of the form $\exists y^\mj\psi$ and $\mathcal{D}$ ends with an application of $\exists I$ to a shorter derivation of $\Gamma\vdash_\mathrm{MSL}\psi[t^\mj/y^\mj]$. By the induction hypothesis,
$$\Gamma^*,J,\{Q_\mi x^\mi:x^\mi\in \mathrm{FV}(\Gamma,\psi[t/y])\}\vdash_{\mathrm{FOL}}\psi^*[t/y].$$
Since $\mathrm{FV}(t)\subseteq \mathrm{FV}(\Gamma,\psi[t/y])$, we have $J,\{Q_\mi x^\mi:x^\mi\in \mathrm{FV}(\Gamma,\psi[t/y])\}\vdash_\mathrm{FOL} Q_\mj t$. By ${\land}I$ and ${\exists}I$,
$$\Gamma^*,J,\{Q_\mi x^\mi:x^\mi\in \mathrm{FV}(\Gamma,\psi[t/y])\}\vdash_{\mathrm{FOL}}\exists y(Q_\mj y\land \psi^*)=\varphi^*.$$
By Lemma~\ref{RemoveQAssumptions},
$$\Gamma^*,J,\{Q_\mi x^\mi:x^\mi\in \mathrm{FV}(\Gamma,\varphi)\}\vdash_{\mathrm{FOL}}\varphi^*.$$
 
\medskip
 
\noindent\textbf{Case $\exists E$:}
Suppose that the last rule was $\exists E$. Then $\mathcal{D}$ ends with an application of $\exists E$ to shorter derivations of ${\Gamma_1\vdash_\mathrm{MSL}\exists y^\mj\psi}$ and $\Gamma_2,\psi\vdash_\mathrm{MSL}\varphi$, where $\Gamma_1\cup\Gamma_2\subseteq \Gamma$ and  $y^\mj$ does not occur free in $\Gamma_2$ or $\varphi$. By the induction hypothesis,
$$\Gamma_1^*,J,\{Q_\mi x^\mi:x^\mi\in \mathrm{FV}(\Gamma_1,\exists y^\mj\psi)\}\vdash_{\mathrm{FOL}}\exists y(Q_\mj y\land \psi^*)$$
and
$$\Gamma_2^*,\psi^*,J,\{Q_\mi x^\mi:x^\mi\in \mathrm{FV}(\Gamma_2,\psi,\varphi)\}\vdash_{\mathrm{FOL}}\varphi^*.$$
The latter gives
$$\Gamma_2^*,Q_\mj y\land\psi^*,J,\{Q_\mi x^\mi:x^\mi\in \mathrm{FV}(\Gamma_2,\psi,\varphi)\}\setminus \{Q_\mj y\}\vdash_{\mathrm{FOL}}\varphi^*.$$
We can now apply $\exists E$ on $\exists y(Q_\mj y\land \psi^*)$ to get
$$\Gamma^*,J,\{Q_\mi x^\mi:x^\mi\in \mathrm{FV}(\Gamma,\varphi)\}\vdash_{\mathrm{FOL}}\varphi^*.$$

\end{proof}

\begin{corollary}
If $\Gamma\cup\{\varphi\}$ is a set of $\mathcal{L}$-sentences, then
$$
\Gamma\vdash_{\mathrm{MSL}}^{\mathcal{L}}\varphi\quad\text{implies}\quad \Gamma^*,J\vdash_{\mathrm{FOL}}\varphi^*. 
$$
\end{corollary}

\section{Proof of Proposition~\ref{LibeqExt}}\label{app:libeq}

\begin{lemma}\label{LibeqLiberalRel}
Let $\mathcal{L}^+$ be obtained from $\mathcal{L}$ by making all relation symbols liberal. Then for all $\mathcal{L}^\approx$-formulas $\Gamma\cup\{\varphi\}$,
$$\Gamma\vdash_{\mathrm{MSL}^\approx}^{\mathcal{L}}\varphi\quad\text{iff}\quad\Gamma\vdash_{\mathrm{MSL}^\approx}^{\mathcal{L}^+}\varphi.$$
\end{lemma}
\begin{proof} We only prove the right-to-left direction. Fix a derivation $\mathcal{D}$ of $\Gamma\vdash_{\mathrm{MSL}^\approx}^{\mathcal{L}^+}\varphi$, and let $S$ be the set of sorts appearing in $\mathcal{D}$. For $R\in \mathcal{L}$ with arity $n$ and terms $t_1,\dots,t_n$ of any sorts, let
$$\zeta_R(t_1,\dots, t_n):=\bigvee_{(\mi_1, \dots, \mi_n)\in \mathrm{rank}(R)\cap S^n}\exists y_1^{\mi_1}\ldots \exists y_n^{\mi_n}\left(t_1\approx y_1\land \ldots \land t_n\approx y_n\land R(y_1,\ldots, y_n)\right)$$
where each $y_{\mi_k}$ is a fresh variable of sort $\mi_k$. Then $\zeta_R(t_1,\dots, t_n)$ is a well-sorted $\mathcal{L}^\approx$-formula, and
$$\vdash_{\mathrm{MSL}^\approx}^\mathcal{L}R(s_1^{\mi_1},\ldots, s_n^{\mi_n})\leftrightarrow \zeta_R(s_1^{\mi_1},\ldots, s_n^{\mi_n})$$
for all $\mathcal{L}$-terms $s_1^{\mi_1},\ldots, s_n^{\mi_n}$ with $(\mi_1,\ldots, \mi_n)\in \mathrm{rank}(R)\cap S^n.$ 
\smallskip

Now, replace each instance of $R(t_1,\ldots, t_n)$ in $\mathcal{D}$ with $\zeta_R(t_1,\dots, t_n)$. We get an $\mathcal{L}^\approx$-derivation of a formula that is $\mathcal{L}^\approx$-equivalent to $\varphi$ from formulas that are $\mathcal{L}^\approx$-equivalent to $\Gamma$ together with a set of formulas that are derivable from $\mathrm{Eq}^\approx_\mathcal{L}$.

\end{proof}

\libeqext*

\begin{proof} By Lemma~\ref{LibeqLiberalRel}, we can assume that all relation symbols of $\mathcal{L}$ and $\mathcal{L}'$ are liberal. We let $(\cdot)^\circ$ be defined on terms as in the proof of Proposition~\ref{ConLan}. We fix a fresh variable $w^{\mi_0}$ with $\mi_0\in \mathrm{Sort}$ and extend $(\cdot)^\circ$ to formulas as follows:
\begin{align*}
  \bot^\circ &:= \bot, \\
  (t_1^\mi = t_2^\mi)^\circ &:=
    \begin{cases}
      t_1^\circ = t_2^\circ & \text{if } \mi\in\mathrm{Sort}, \\
      \top & \text{else;}
    \end{cases} \\
    (t_1^\mi\approx t_2^\mj)^\circ & :=\begin{cases}
t_1^\circ \approx t_2^\circ &\text{if } \mi\in \mathrm{Sort}\text{ and }\mj\in \mathrm{Sort},\\
\top &\text{if }\mi=\mj\notin \mathrm{Sort},\\
\bot &\text{else;}
\end{cases}\\
  \bigl(R(t_1^{\mi_1},\dots,t_n^{\mi_n})\bigr)^\circ &:=
    \begin{cases}
      R(t_1^\circ,\dots,t_n^\circ) & \text{if $R\in\mathrm{Rel}$ and $(\mi_1,\dots \mi_n)\in \mathrm{rank}(R)$,} \\
      \top & \text{else;}
    \end{cases} \\
  (\psi\star\chi)^\circ &:= \psi^\circ\star\chi^\circ
    \quad\text{for } \star\in\{\land,\lor,\to\}, \\
  (\exists x^\mi\,\psi)^\circ &:=
    \begin{cases}
      \exists x^\mi\,\psi^\circ & \text{if } \mi\in\mathrm{Sort}, \\
      \exists w^{\mi_0}\psi^\circ              & \text{else;}
    \end{cases} \\
  (\forall x^\mi\,\psi)^\circ &:=
    \begin{cases}
      \forall x^\mi\,\psi^\circ & \text{if } \mi\in\mathrm{Sort}, \\
      \forall w^{\mi_0}\psi^\circ               & \text{else.}
    \end{cases}
\end{align*}

Again, $(\cdot)^\circ$ maps $(\mathcal{L}')^\approx$-formulas to $\mathcal{L}^\approx$-formulas, $\psi^\circ = \psi$ for all $\mathcal{L}^\approx$-formulas, and $\bigl(\psi[t^\mi/x^\mi]\bigr)^\circ=\psi^\circ[t^\circ/x]$ whenever $x^\mi$ is distinct from $z_\mi$. We can therefore turn $\mathcal{D}$ into an $\mathcal{L}^\approx$-derivation of $\varphi$ from $\Gamma$ together with $\left(\mathrm{Eq}^\approx_{\mathcal{L}'}\right)^\circ$. What remains is to verify that for each $\chi\in \mathrm{Eq}^\approx_{\mathcal{L}'}$, the formula $\chi^\circ$ is derivable from $\mathrm{Eq}^\approx_\mathcal{L}$. We spell out the two congruence schemas.

\smallskip

Consider the case where $\chi$ is:
$$x_1^{\mi_1}\approx y_1^{\mj_1}\land \ldots \land x_n^{\mi_n}\approx y_n^{\mj_n}\rightarrow \left(R(x_1^{\mi_1},\ldots,x_n^{\mi_n})\to R(y_1^{\mj_1},\ldots,y_n^{\mj_n})\right).$$

If $R\notin\mathcal{L}$, then $(R(x_1^{\mi_1},\ldots,x_n^{\mi_n}))^\circ=\top=(R(y_1^{\mj_1},\ldots,y_n^{\mj_n}))^\circ$. So $\chi^\circ$ is a tautology.

Now suppose that $R\in\mathcal{L}$. If, for some $k$, $\mi_k\neq \mj_k$ and $\mi_k\notin \mathrm{Sort}$ or $\mj_k\notin \mathrm{Sort}$, then $(x_k^{\mi_k}\approx y_k^{\mj_k})^\circ =\bot$. So $\chi^\circ$ is a tautology. If $\mi_k=\mj_k\notin \mathrm{Sort}$ for some $k$, then $(\mi_1,\ldots, \mi_n)\notin \mathrm{rank}(R)\subseteq \mathrm{Sort}^n,$
so $(R(x_1^{\mi_1},\ldots,x_n^{\mi_n}))^\circ=\top=(R(y_1^{\mj_1},\ldots,y_n^{\mj_n}))^\circ.$ Again, $\chi^\circ$ is a tautology. Finally, if $\mi_k,\mj_k\in \mathrm{Sort}$ for each $k$, then $\chi^\circ$ is an instance of the congruence schema in $\mathrm{Eq}^\approx_\mathcal{L}$.

\medskip

Now consider the case where $\chi$ is:
$$x_1^{\mi_1}\approx y_1^{\mj_1}\land \ldots \land x_n^{\mi_n}\approx y_n^{\mj_n}\rightarrow f(x_1^{\mi_1},\ldots,x_n^{\mi_n})\approx f(y_1^{\mj_1},\ldots, y_n^{\mj_n}).$$
Since $f$ is strict, we have that $\mi_k=\mj_k$ for each $k$. If $f\notin \mathcal{L}$, then both $(f(x_1^{\mi_1},\ldots,x_n^{\mi_n}))^\circ$ and $(f(y_1^{\mi_1},\ldots,y_n^{\mi_n}))^\circ$ equal $z_\mi$, where $(\mi_1,\dots,\mi_n,\mi)\in \mathrm{rank}'(f)$. If $\mi\in \mathrm{Sort}$, then $(f(x_1^{\mi_1},\ldots,x_n^{\mi_n})\approx f(y_1^{\mi_1},\ldots,y_n^{\mi_n}))^\circ$ is $z_\mi\approx z_\mi$. If $\mi\notin \mathrm{Sort}$, it is $\top$. In either case, $\chi^\circ$ is derivable from $\mathrm{Eq}^\approx_\mathcal{L}$.

If $f\in\mathcal{L}$, then $\mathrm{rank}'(f)=\mathrm{rank}(f)=(\mi_1,\ldots,\mi_n,\mi)$ with $\mi_1,\ldots,\mi_n,\mi\in\mathrm{Sort}$, since $f$ is strict in both $\mathcal{L}$ and $\mathcal{L}'$. So $\chi^\circ=\chi$.

\end{proof}

\section{Proof of Claim~\ref{FormSubst}}\label{app:FormSubst}
\FormSubst*
\begin{claimproof}[Proof of Claim]
We start by showing by induction on $t$ that for any term $s$, variable $x$, and sort $\ml$ that
$$(t[s/x])^{\langle \ml\rangle} = t^{\langle \ml\rangle}[s^{\langle \mi_1\rangle}/x^{\langle \mi_1\rangle}, \ldots, s^{\langle \mi_m\rangle}/x^{\langle \mi_m\rangle}].$$
If $t=y$ for some variable different from $x$, then $y^{\langle \ml\rangle} \neq x^{\langle \mi_k\rangle}$ for each $k$, so
    $$
        (t[s/x])^{\langle \ml\rangle}=(y[s/x])^{\langle \ml\rangle}=y^{\langle \ml\rangle}
        =y^{\langle \ml\rangle}[s^{\langle \mi_1\rangle}/x^{\langle \mi_1\rangle}, \ldots, s^{\langle \mi_m\rangle}/x^{\langle \mi_m\rangle}]
        =t^{\langle \ml\rangle}[s^{\langle \mi_1\rangle}/x^{\langle \mi_1\rangle}, \ldots, s^{\langle \mi_m\rangle}/x^{\langle \mi_m\rangle}].$$
If $t=x$, then $t^{\langle \ml\rangle}=x^{\langle \ml\rangle}$ and $\ml=\mi_k$ for some $k$, so
$$
    (t[s/x])^{\langle \ml\rangle}=(x[s/x])^{\langle \ml\rangle}=s^{\langle \ml\rangle}
    =x^{\langle \ml\rangle}[s^{\langle \mi_1\rangle}/x^{\langle \mi_1\rangle}, \ldots, s^{\langle \mi_m\rangle}/x^{\langle \mi_m\rangle}]
    =t^{\langle \ml\rangle}[s^{\langle \mi_1\rangle}/x^{\langle \mi_1\rangle}, \ldots, s^{\langle \mi_m\rangle}/x^{\langle \mi_m\rangle}].$$
For the induction step, suppose $t=f(t_1,\ldots, t_n)$ with $\mathrm{rank}(f)=(\mr_1,\ldots,\mr_n,\mr)$. Then
\begin{align*}
    (t[s/x])^{\langle \ml\rangle}&=(f(t_1[s/x],\ldots, t_n[s/x]))^{\langle \ml\rangle}\\
    &=f^{\langle \ml\rangle}((t_1[s/x])^{\langle \mr_1\rangle},\ldots,(t_n[s/x])^{\langle \mr_n\rangle})\\
    &=f^{\langle \ml\rangle}(t_1^{\langle \mr_1\rangle}[s^{\langle \mi_1\rangle}/x^{\langle \mi_1\rangle}, \ldots, s^{\langle \mi_m\rangle}/x^{\langle \mi_m\rangle}],\ldots,t_n^{\langle \mr_n\rangle}[s^{\langle \mi_1\rangle}/x^{\langle \mi_1\rangle}, \ldots, s^{\langle \mi_m\rangle}/x^{\langle \mi_m\rangle}])\\
    &=t^{\langle \ml\rangle}[s^{\langle \mi_1\rangle}/x^{\langle \mi_1\rangle}, \ldots, s^{\langle \mi_m\rangle}/x^{\langle \mi_m\rangle}].
\end{align*}

We now prove the claim by induction on the complexity of $A$. The atomic cases follow from what we just showed, and the cases for the propositional connectives are immediate. Now suppose that $A=\exists yB$ and that $t$ is free for $x$ in $\exists yB$. Then, either $x$ is not free in $A$, or $y$ does not occur in $t$ and $t$ is free for $x$ in $B$.

If $x$ is not free in $A$, then no $x^{\langle \mi_k\rangle}$ is free in $A^\diamond$. Thus, each $t^{\langle \mi_k\rangle}$ is free for $x^{\langle \mi_k\rangle}$ in $A^\diamond$, and $(A[t/x])^\diamond=A^\diamond=A^\diamond[t^{\langle \mi_1\rangle}/x^{\langle \mi_1\rangle}, \ldots, t^{\langle \mi_m\rangle}/x^{\langle \mi_m\rangle}].$

If $y$ does not occur in $t$ and $t$ is free for $x$ in $B$, then none of $y^{\langle \mj_1\rangle},\ldots, y^{\langle \mj_p\rangle}$ occurs in any $t^{\langle \mi_k\rangle}$. By the induction hypothesis, each $t^{\langle \mi_k\rangle}$ is free for $x^{\langle \mi_k\rangle}$ in $B^\diamond$ and $(B[t/x])^\diamond= B^\diamond[t^{\langle \mi_1\rangle}/x^{\langle \mi_1\rangle}, \ldots, t^{\langle \mi_m\rangle}/x^{\langle \mi_m\rangle}]$. Thus, each $t^{\langle \mi_k\rangle}$ is free for $x^{\langle \mi_k\rangle}$ in $A^\diamond$, and 
    $$(A[t/x])^\diamond= \exists y^{\langle \mj_1\rangle}\ldots\exists y^{\langle \mj_p\rangle} \bigl(B^\diamond[t^{\langle \mi_1\rangle}/x^{\langle \mi_1\rangle}, \ldots, t^{\langle \mi_m\rangle}/x^{\langle \mi_m\rangle}]\bigr).$$
    Since no $y^{\langle \mj\rangle}$ occurs in any $t^{\langle \mi_k\rangle}$, 
    \begin{align*}
        \exists y^{\langle \mj_1\rangle}\ldots\exists y^{\langle \mj_p\rangle} \bigl(B^\diamond[t^{\langle \mi_1\rangle}/x^{\langle \mi_1\rangle}, \ldots, t^{\langle \mi_m\rangle}/x^{\langle \mi_m\rangle}]\bigr)&=\bigl(\exists y^{\langle \mj_1\rangle}\ldots\exists y^{\langle \mj_p\rangle} B^\diamond\bigr)[t^{\langle \mi_1\rangle}/x^{\langle \mi_1\rangle}, \ldots, t^{\langle \mi_m\rangle}/x^{\langle \mi_m\rangle}]\\
        &=A^\diamond[t^{\langle \mi_1\rangle}/x^{\langle \mi_1\rangle}, \ldots, t^{\langle \mi_m\rangle}/x^{\langle \mi_m\rangle}].
    \end{align*}
    The case where $A=\forall y\, B$ is identical.
    
\end{claimproof}

\end{document}